\documentclass[12pt,reqno,fleqn]{amsart}  
\usepackage{tikz}
\usepackage{amsmath,amstext,amsthm,amssymb,amsxtra}
\usepackage{mathrsfs}%for special fonts
\usepackage[top=1.5in, bottom=1.5in, left=1.25in, right=1.25in]	{geometry}
\usepackage{lmodern}

\usepackage{graphics}
\usepackage{euscript}
\usepackage{xcolor}
\usepackage{mathtools}
\mathtoolsset{showonlyrefs,showmanualtags}

\usepackage{hyperref} %,hypertexnames=false,colorlinks,[pagebackref]
\hypersetup{
	colorlinks=true,       % false: boxed links; true: colored links
	linkcolor=blue,          % color of internal links
	citecolor=magenta,        % color of links to bibliography
	filecolor=magenta,      % color of file links
	urlcolor=cyan           % color of external links
}

\usepackage[msc-links]{amsrefs}

\newtheorem{theorem}{Theorem}[section]

\newtheorem{proposition}{Proposition}[section]
\newtheorem{lemma}{Lemma}[section]
\newtheorem{corollary}{Corollary}[section]
\newtheorem{OldTheorem}{Theorem}
\theoremstyle{definition}
\newtheorem{definition}{Definition}[section]
\theoremstyle{definition}

\newtheorem{remark}{Remark}[section]

\theoremstyle{remark}
\newtheorem{prop}{U}

\def\supp{{\rm supp\,}}

\def\dist{{\rm dist}}
\def\Int{{\rm Int}}

\def\ZS{\ensuremath{\mathcal S}}

\def\ZI{\ensuremath{\textbf 1}}

\def\ZN{\ensuremath{\mathbb N}}

\def\ZD{\ensuremath{\cal D}}

\def\ZR{\ensuremath{\mathbb R}}

\def\ZT{\ensuremath{\mathbb T}}

\def\ZD{\ensuremath{\mathcal D}}

\numberwithin{equation}{section}
\def\md#1#2\emd{\ifx0#1
	\begin{equation*} #2 \end{equation*}\fi  %  single line display, no number
	\ifx1#1\begin{equation}#2\end{equation}\fi   % single line display, number
	\ifx2#1\begin{align*}#2\end{align*}\fi   % aligned display, no number
	\ifx3#1\begin{align}#2\end{align}\fi    % aligned display, number
	\ifx4#1\begin{gather*}#2\end{gather*}\fi  % multline, not align, no number
	\ifx5#1\begin{gather}#2\end{gather}\fi   % multinline, not align
	\ifx6#1\begin{multline*}#2\end{multline*}\fi  %  display too long for one line
	\ifx7#1\begin{multline}#2\end{multline}\fi  % as above, with numbers
	\ifx8#1\begin{multline*}\begin{split}#2\end{split}\end{multline*}\fi
	\ifx9#1\begin{multline}\begin{split}#2\end{split}\end{multline}\fi
}
\newcommand {\e }[1]{\eqref{#1}}
\newcommand {\lem }[1]{Lemma \ref{#1}}

\newcommand {\cor }[1]{Corollary \ref{#1}}
\newcommand {\pro }[1]{Proposition \ref{#1}}
\newcommand {\trm }[1]{Theorem \ref{#1}}

\newcommand {\df }[1]{Definition \ref{#1}}
\newcommand {\p }[1]{{U\ref{#1}}}

\title[] {On the convergence sets of operator sequences \\on spaces of homogeneous type}
\author{Grigori A. Karagulyan}
\address{Institute of Mathematics of NAS of RA, Marshal Baghramian ave., 24/5, Yerevan, 0019, Armenia} 
\address{Faculty of Mathematics and Mechanics, Yerevan State
	University, Alex Manoogian, 1, 0025, Yerevan, Armenia} 
\email{g.karagulyan@ysu.am}

\thanks{The work was supported by the Science Committee of RA, in the frames of the research project 21AG‐1A045}

\subjclass[2010]{42C05, 42C10, 42C25, 42A55}
\keywords{Calder\'on-Zygmund operator, weighted inequalities, sparse operators, maximal function, martingale transform, homogeneous spaces}
\begin{document}

	\begin{abstract}
		We consider sequences of operators $U_n:L^1(X)\to M(X)$, where $X$ is a space of homogeneous type. Under certain conditions on the operators $U_n$ we give a complete characterization of convergence (divergence) sets of functional sequences $U_n(f)$, where $f\in L^p(X)$, $1\le p\le \infty$. The results are applied to characterize convergence sets of some specific operator sequences in classical analysis.
	\end{abstract}

	\maketitle  
	%%%%%%%%%%%%%%%%%%%%%%%%%%%%%% SECTION  SECTION SECTION
	%%%%%%%%%%%%%%%%%%%%%%%%%%%%%% SECTION  SECTION SECTION
	\section{Introduction}
	\subsection{A historical overview}
	Let 
	\begin{equation}\label{x100}
		f=\{f_k(x),\quad k=1,2,\ldots\}
	\end{equation} 
be an infinite sequence of real functions. Denote by $C(f)\subset \ZR$ the convergence set of sequence \e{x100}, i.e. the set of points $x\in \ZR$ such that $\lim f_n(x)$ exists. A classical theorem of Hahn-Sierpinski \cite{Hahn,Sie} asserts that if functions \e{x100} are continuous, then $C(f)$ is a $F_{\sigma\delta}$-set, and conversely, every $F_{\sigma\delta}$-set is a convergence set for a sequence of continuous functions.  The first part of this statement is immediate, since the convergence set of \e{x100} may be written in the form
	\begin{equation*}
	C(f)=\cap_{m=1}^\infty\cup_{n=1}^\infty\cap_{k=1}^\infty \{x:\, |f_{n+k}(x)-f_n(x)\le 1/m\},
	\end{equation*}
and the latter is a $F_{\sigma\delta}$-set if the functions are continuous (see \df{D1}). The second part of the Hahn-Sierpinski theorem requires certain construction of a sequence of continuous functions for which a given $F_{\sigma\delta}$-set is a convergence set (see also \cite{Hau}, page 307). Note that the complement of the convergence set $C(f)$ is the divergence set of \e{x100}, which we denote by $D(f)$. So the $G_{\delta\sigma}$ sets completely  characterize the divergence sets of sequences of continuous functions. One can also consider unbounded divergence set $UD(f)\subset D(f)$ of \e{x100}, which consists of the points $x$, satisfying $\limsup |f_n(x)|=\infty$. It is also known that for a set to be the $UB$-set for a sequence of continuous functions it is necessary and sufficient to be a $G_\delta$-set (see \cite{Hahn,Sie,Hau}). Hence, we have certain topological characterization of the $C$, $D$ and $UD$ sets in the class of sequences consisting of continuous functions.

Such characterization problems were commonly considered in many fields of analysis (Fourier series, analytic functions on the unit ball, power series, differentiation theory) and there are many published papers and open problems, some of which will be considered in the last section. Earlier reviews of problems, concerning  the convergence or divergence sets of Fourier series in trigonometric, Walsh and Haar orthogonal systems the reader can find in papers \cite{Wade,Uly1,Uly2}. We also refer the monograph \cite{CoLo} related to boundary exceptional sets of analytic functions. 

Let us state few remarkable examples of characterization theorems below. The first one provides a complete characterization of non-differentiability points sets of continuous functions. 
\begin{OldTheorem}[Zahorski, \cite{Zah}]\label{OT3}
	A set $E\subset \ZR$ can be a set of non-differentiability points of a real continuous function if and only if it is a union of a $G_\delta$-set and a $G_{\delta\sigma}$-null-set.
\end{OldTheorem}
  In fact, the main ingredient of this theorem is the construction of a continuous function, whose non-differentiability points set coincides with a given set, which is a union of a $G_\delta$-set and a $G_{\delta\sigma}$-set of measure zero.  For the realization of this construction in \cite{Zah} a subtle technique was used (see also \cite{Pir} and \cite{Fow} for alternative simplified proofs of \trm{OT3}). The next well-known result is an extension of Kolmogorov's celebrated theorem of \cite{Kol} on an example of an integrable function, whose Fourier series is everywhere divergent.
\begin{OldTheorem}[Zeller, \cite{Zel}]\label{T-Z}
	Let $E$ be a $F_\sigma$-set in the circle $\ZT=\ZR/2\pi$. Then there exists a function $f\in L^1(\ZT)$ whose Fourier series convergence at any point $x\in E$ and unboundedly diverges whenever $x\in \ZT\setminus E$.
\end{OldTheorem}
This theorem, combining with above mentioned general results, provides a complete characterization of $UD$-sets in the class of Fourier series of functions from $L^1(\ZT)$. Later, an analogous result for the Walsh series was proved in \cite{Luk1, Luk2}. Note that it is an open problem to characterize $D$ or $UD$ sets of Fourier series of $L^p(\ZT)$, $1<p\le \infty$, or continuous functions (see \cite{Uly1}, page 107).  Note that according to Carleson-Hunt \cite{Car,Hunt} results, Fourier series of functions from $L^p(\ZT)$, $p>1$ converge a.e., so the devergence sets of those series are null-sets. Some partial results concerning the characterization of convergence or divergence sets of Fourier series one can find in \cite{Bug,Buz,Buz1,Tai,Gog,Khe, KaKa,Kar3,Kar4,KarKar,Lun, KaGa}. Mostly, those papers provide construction of different examples of Fourier series, diverging on a given null-set and
in this context the most general is Kahane-Katznelson's \cite{KaKa} theorem on existence of a continuous function, whose trigonometric Fourier series diverges on a given null-set (see also \cite{Kat}, chap. II.3, for a detailed discussion of divergence sets). We also refer papers \cite{Kar3,Lun}, which provide complete characterization of $C$, $D$ and $UD$ sets of Fourier-Haar series. The author in \cite{Kar2} obtained a general characterization theorems for certain operator sequences, which cover many results of above mentioned papers, providing complete characterization of convergence sets of Fourier series in Haar, Franklin systems, as well as positive order Cezaro means of trigonometric and Walsh Fourier series. We remark here that K\"orner \cite{Kor} constructed a $G_\delta$-set, which is not a convergence set for any kind of trigonometric series. This example says that for the ordinary partial sums of trigonometric or Walsh Fourier series a pure topological characterization of convergence sets in some open problems may fail (see discussion in \cite{Uly1,Uly2}). 

The next result provides a complete characterization of radial convergence sets of bounded analytic functions on the unit disc. It was a solution of a longstanding problem posed by Collingwood and Lohwater in \cite{CoLo}. Namely,
\begin{OldTheorem}[Kolesnikov, \cite{Kol}]
	Let $D$ be the unit disc with the boundary circle $\Gamma$. For a set $E\subset \Gamma$ to be the radial convergence set for a bounded analytic function on $D$ it is necessary and sufficient to be a $F_{\sigma\delta}$ set of full measure.
\end{OldTheorem}

\subsection{Main results of the paper}

In the present paper we obtain complete characterization theorems for certain operator sequences living in general spaces of homogeneous type. Those generalize the results of \cite{Kar1,Kar2}, where the operators were considered on the interval $[0,1]$. By a quasi-distance on a set $X$ we mean a non-negative function $d(x,y)$ defined on $X\times X$ such that
\begin{enumerate}
	\item $d(x,y)=0$ if and only if $x=y$,
	\item $d(x,y)=d(y,x)$ for any $x,y\in X$,
	\item $d(x,y)\le K(d(x,z)+d(z,y))$,
\end{enumerate}
where $K>0$ is a constant. Such a distance defines a topology on $X$, for which the balls $B(x,r)=\{y\in X:\, d(y,x)<r\}$ form a base. Namely, a set $G\subset X$ in this topology is open if and only if for any point $x\in G$ there exists a ball $B(x,r)\subset G$. However, the balls themselves need not to be open when $K>1$.  Let $\mu$ be measure on a $\sigma$-algebra in $X$, containing all Borel sets and balls, such that 
\begin{equation}
	\mu(B(x,2r))\le C\mu(B(x,r))<\infty.
\end{equation}
Such a combination $(X,d,\mu)$ is called a space of homogeneous type, which definition is adapted from \cite{CoWe} (see also \cite{Christ}). 

 \begin{definition}\label{D1}
 	Recall that a set in a topological space is of $G_\delta$ type provided it is a countable intersection of open sets, and a set is $G_{\delta\sigma}$ if it is a countable union of $G_\delta$-sets. The $F_\sigma$ sets are the countable unions of closed sets, and $F_{\sigma\delta}$ are the countable intersections of $F_\sigma$ sets. 
 \end{definition}
\noindent
\begin{definition}
	The distance of two sets $A$ and $B$ in a space of homogeneous type $X$ will be denoted by
	\begin{equation}
		\dist(A,B)=\inf_{x\in A,\,y\in B} d(x,y)
	\end{equation} 
	The notation $A\Subset B$ will stand for the relation $\dist(A,B^c)>0$. Note that $A\Subset B$ implies $B^c\Subset A^c$.

\end{definition}
Let $L^1(X)$ denote the space of Lebesgue integrable functions on a space of homogeneous type $(X,d,\mu)$ and suppose that $M(X)$ is  the normed space of bounded measurable functions on $X$ with the norm $\|f\|_M=\sup_{x\in X}|f(x)|$.
We consider sequences of linear operators
\begin{equation}\label{1}
	U_n:L^1(X)\to M(X),\quad n=1,2,\ldots,
\end{equation}
which may have some of the following properties

\begin{prop}\label{C1}
	$\rho_n=\|U_n\|_{L^1\to M}<\infty$, $n=1,2,\ldots$,
\end{prop}
\begin{prop}\label{C2}
	$\varrho=\sup_n\|U_n\|_{L^\infty \to M}<\infty$,
\end{prop}
\begin{prop}\label{C3}
if $f\in L^1(X)$ is constant on an open set $G\subset X$, then $U_n(x,f)$ uniformly converges over any set $E\Subset G$.
\end{prop}
\begin{prop}\label{C4}
		if $G\subset X$ is an open set with $\mu(G)<\infty$, then $U_n(x,\ZI_G)\to \ZI_G(x)$ almost everywhere as $n\to\infty$.
\end{prop}

\noindent
We will always suppose that $X$ is a space of homogeneous type unless other assumptions are suggested. The main results of the paper read as follows.
\begin{theorem}\label{T1}
	If an operator sequence \e{1} satisfies \p{C1}-\p{C4}, then for any $G_{\delta\sigma}$ null-set $E\subset X$  and any $\varepsilon>0$, there exists a function $f\in L^\infty(X)$ such that 
	
	\vspace{2mm}
	1) $\mu(\supp f)<\varepsilon$, 
	
	\vspace{2mm}
	2) $U_n(x,f)$ diverges at any $x\in E$,
	
	\vspace{2mm}
	3) $U_n(x,f)\to f(x)$ if $x\in X\setminus E$.
	
\end{theorem}

\begin{theorem}\label{T2}
Let an operator sequence \e{1} satisfy \p{C1}-\p{C4}, and for a function $\psi:[0,\infty)\to [0,\infty)$ we have
\begin{equation}
	\lim_{t\to\infty}\psi(t)=\infty.
\end{equation}
Then for every $G_{\delta}$ null-set $E\subset X$ and $\varepsilon>0$, there exists a measurable function $f$ such that

\vspace{2mm}
1) $\mu(\supp f)<\varepsilon, \quad \int_X\psi(|f|)<\infty$,

\vspace{2mm}
2) $\limsup_{n\to\infty}|U_n(x,f)|=\infty$ for all $x\in E$,

\vspace{2mm}
3) $U_n(x,f)\to f(x)$ if $x\in X\setminus E$.
\end{theorem}
	Observe that if besides of properties \p{C1}-\p{C4} we also require the continuity of $U_n(x,f)$ as functions on $x$, then \trm{T1} gives the complete characterization of $C$ and $D$ sets of sequences 
	$U_n(x,f)$, where $f\in L^p(X)$ and $1\le p\le \infty$. Similarly, \trm{T2}  provides complete characterization of $UD$-sets of those sequences when $1\le p<\infty$.
	One can not claim for the function $f$ in \trm{T2} to be bounded like in \trm{T1}, since in that case condition \p{C2} would imply the boundedness of $U_n(x,f)$ for any $x\in X$.

In the case when $X$ coincides with $[0,1]$, Theorems \ref{T1}  and \ref{T2} were proved in \cite{Kar2}, and we essentially use some arguments of \cite{Kar1, Kar2} in the present work. Note that in a recent paper \cite{KaGa} authors consider kernel operator sequences 
\begin{equation*}
	U_n(x,f)=\int_{[0,1]^d}K_n(x,t)f(t)dt
\end{equation*}
on the $d$-dimensional cube $[0,1]^d$ (or $\ZR^d$), where the kernels $K_n(x,t)$ satisfy certain conditions in the spirit of an approximation of identity. Those conditions provide properties \p{C1}-\p{C3} for the operators $U_n$, but instead of \p{C4} the paper requires the convergence of $U_n(x,f)$ at any continuity point $x$ of $f$. Under these requirements on $U_n$ the main result of \cite{KaGa} proves that any countable set is a divergence set for a sequence $U_n(x,f)$ with a function $f\in M[0,1]^d$.

The approach provided in the proofs of Theorems \ref{T1} and \ref{T2} enables to obtain also the following pure divergence result, where the operator sequences satisfy only conditions \p{C1} and \p{C3}.  This result is a generalization of an analogous theorem of \cite{Kar1} for operators living on $[0,1]$.
\begin{theorem}\label{T3}
	If an operator sequence \e{1} satisfies \p{C1} and \p{C3}, then for every null-set $E\subset X$, there exists a set $G$ such that for the indicator function $f=\ZI_G$ we have a divergence of  $U_n(x,f)$ at any point $x\in E$. Moreover, 
	\begin{equation}
		\limsup_{n\to\infty}U_n(x,f)\ge 1,\quad \liminf_{n\to\infty}U_n(x,f)\le  0\text { for all }x\in E.
	\end{equation}
\end{theorem}

\begin{remark}
	With a misery change in the proofs, Theorems \ref{T1}-\ref{T3} can be stated and proved for operators $U_r$, where the parameter $r$ varies in an infinite partially ordered set $R$. Namely, we can consider a partially ordered set $R$ with a relation $<$, satisfying
	
	\vspace{2mm}
	1) there is a unique maximal element $r_\infty\in R$,
	
	\vspace{2mm}
	2) for any $r<r_\infty$ there exists an infinitely many elements $s$, satisfying $s>r$.
	
	\vspace{2mm}
	Then given a process (or a sequence) $x_r$, $r\in R$, in $X$, the convergence $\lim_{r\to r_\infty} x_r=a$ requires the following: for any open set $G\supset a$ there exists a $r\in R$ such that for any $s$ with $r<s<r_\infty$ we have $x_r\in G$. 
\end{remark}
We note that \trm{T3} as well as the result of paper \cite{Kar1} generalize the results of papers \cite{Khe,Bug,Gog,Luk1,Luk2,Pro}, where authors consider different Fourier partial sums (trigonometric, Walsh, Haar) instead of operators $U_n$. Examples of operators satisfying \p{C1}-\p{C4} and corollaries of Theorems \ref{T1} and \ref{T2} will be considered in the next section. 

\section{Examples and applications}
\subsection{Approximation of the identity on metric measure spaces}
Let $X$ be a space of homogeneous type. A sequence of kernel functions $K_n(x,y)\in L^\infty(X\times X)\cap L^1(X\times X)$ is said to be an approximation of the identity if it satisfies the conditions
\begin{align}
	&\int_X K_n(x,y)dy\rightrightarrows1 \text{ as } n\to\infty,\label{x54}\\
	&\int_{\{y:\,d(x,y)>\delta\}} K_n^*(x,y)dy\rightrightarrows0\text{ as } n\to\infty,\text{ for every } \delta>0 ,\label{x55}\\
	&\sup_{n,x}\int_X |K_n^*(x,y)|dy<\infty.\label{x60}
\end{align}
where the convergence in \e{x54} and \e{x55} holds uniformly with respect to $x$ over $X$, and
\begin{equation*}
	K_n^*(x,y)=\sup_{t:\, d(x,t)\ge d(x,y)} |K_n(x,t)|.
\end{equation*}
Consider the operator sequence
\begin{equation}\label{x61}
	U_n(x,f)=\int_XK_n(x,y)f(y)d\mu(y),
\end{equation}
where the  kernels $K_n$ satisfy above conditions. It is known that if a space of homogeneous type satisfies the following property that the space of functions 
\begin{equation}
	C_{K}(X)=\left\{f\in C(X):\,\supp(f)\text{ is bounded and } \sup_{x\in X}|f(x)|<\infty \right\}
\end{equation}
is dense in $L^1(X)$, then operators \e{x61} enjoy the following properties (see for example \cite{CoWe,Christ}):

\vspace{2mm}
	1) If a function $f\in L^1(X)$ is uniformly continuous (in particular is constant) on an open set $G\subset X$, then $U_n(x,f)\rightrightarrows f(x)$ uniformly on any subset $E\Subset G$. 
	
	\vspace{2mm}
	2) If $f\in L^1(X)$, then $U_n(x,f)\to f(x)$ almost everywhere.

\vspace{2mm}
Obviously, such operators satisfy conditions \p{C1}-\p{C4}. Properties \p{C1} and \p{C2} immediately follow from the definition of approximation of the identity, while properties \p{C3} and \p{C4} are weakened cases of above properties 1) and 2) respectively. So we can state the following
\begin{proposition}\label{P}
	If a space of homogeneous type $X$ is such that $C_K(X)$ is dense in $L^1(X)$, then operators \e{x61} obey the conditions of Theorems \ref{T1}-\ref{T3}. 
\end{proposition}

\subsection{Walsh functions} We  consider Walsh orthonormal system, defined on the set of all sequences 
\begin{equation}\label{x51}
(x_0,x_1,x_2,\ldots),\quad x_j=0 \text{ or } 1, \quad j=0,1,2,\ldots.
\end{equation}
(see \cite{GES}, chap. 1). Every such a sequence generates a series
\begin{equation}\label{x62}
	\sum_{j=0}^\infty x_j\cdot 2^{-j-1},
\end{equation} 
which is the dyadic decomposition of some number from $[0,1]$. We note that this correspondence is surjective, but it is not injective, since any dyadic rationals of $[0,1]$ have two decompositions \e{x62}, infinite and finite. For a geometric understanding of the set of sequences \e{x51}  introduce the extended interval $[0,1]^*$, where each dyadic rational number $x$ is doubled, generating  the left point $x-$ corresponding to the finite dyadic decomposition and the right point $x+$ that corresponds to the infinite decomposition of $x$. Writing $x+$ or $x-$ for a dyadic irrational number $x$, we mean just the point $x\in [0,1]^*$. Define a generalized dyadic interval by
\begin{equation}\label{x52}
	[a,b]=\{x+, x-\in [0,1]^*:\, a<x<b\}\cup\{a+,b-\},
\end{equation}
where 
\begin{equation}\label{x50}
	a=\frac{j-1}{2^k},\quad b=\frac{j}{2^k},\quad 1\le j\le 2^k,\quad k=0,1,\ldots.
\end{equation}
One can easily check that the dyadic distance 
\begin{equation}
	d(x,y)=\inf\{b-a:\, [a,b] \text{ is a generalized dyadic interval},\,  x,y\in [a,b]\},
\end{equation}
between two points $x,y\in [0,1]^*$ defines a quasi-distance on $[0,1]^*$. Then any set $E^*\subset [0,1]^*$ has its counterpart $E\subset [0,1]$, obtained by an identifying all the pairs of points $x+,x-\in [0,1]^*$ in one. So we define the measure of $E^*$ to be the Lebesgue measure of the set $E$. Hence $[0,1]^*$ equipped with such a measure becomes a space of homogeneous type. Moreover, one can easily check that $[0,1]^*$ is a compact space (we will not need it). 

To define the Walsh functions, recall the group operation on $[0,1]^*$, defining the sum of two sequences of type \e{x51} $\{x_j\}_{j=0}^\infty$ and  $\{y_j\}_{j=0}^\infty$ to be the sequence $\{z_j\}_{j=0}^\infty$, where 
\begin{equation*}
	z_j=\left\{
	\begin{array}{rcl}
		1&\hbox{ if }& x_j+y_j=1,\\
		0&\hbox{ if }& x_j+y_j=0\hbox{ or }2.
	\end{array}
	\right.
\end{equation*}
Now we can define the Walsh system of functions $\{w_n(x)\}_{n=0}^\infty$ on $[0,1]^*$.  We set $w_0(x)\equiv 1$. For $n\ge 1$ we write its dyadic representation
$n=\sum_{j=0}^k\varepsilon_j2^j,$
then define
\begin{equation}
	w_n(x)=(-1)^{\sum_{j=1}^k \varepsilon_jx_j}\text{ for }x=(x_0,x_1,\ldots)\in [0,1]^*.
\end{equation}
One can see that Walsh functions are continuous in the topology of $[0,1]^*$. Besides, it is well-known that $(C,\alpha)$-means $\sigma_n^\alpha(x,f)$ $\alpha>0$ of the partial sums of Walsh-Fourier series satisfy properties \p{C1}-\p{C4} (see \cite{GES}, chap. 4). So we can state the following results.
\begin{corollary}
	For a set $E\subset [0,1]^*$ to be the divergence (unbounded divergence) set of $\sigma_n^\alpha(x,f)$ for a function $L^p([0,1]^*)$, $1\le p\le \infty$ it is necessary and sufficient to be a $G_{\delta\sigma}$-set ($G_\delta$-set) in the topology of $[0,1]^*$.
\end{corollary}
\begin{remark}
The $(C,\alpha)$-means $\sigma_n^\alpha$ of Walsh-Fourier series can be considered as kernel operators. Nevertheless,  it is known (see \cite{GES}) that the kernels of $\sigma_n^\alpha$ do not form an approximation of the identity.
\end{remark}
\subsection{Splines} Recall the definition of splines on an interval $[a,b]$. For a knot-collection $\Delta = \{t_j\}_{j=1}^{n+k}\subset [a,b]$ such that
\begin{align}
	&t_i\le t_{i+1},\quad t_i\le t_{i+k},\\
	&t_1=\ldots=t_k=a,\quad t_{n+1}=\ldots=t_{n+k}=b
\end{align}
let denote by $\ZS_k(\Delta)$ the space of $k$-order splines with the notes $\Delta$. Those are the functions, which are polynomials of degree $\le k-1$ on each interval $[t_j,t_{j+1}]$ and  has a $k-1-m_j$ continuous derivatives at any node $t_j\in \Delta$ of a multiplicity $m_i$.  Let $P_\Delta$ be the orthoprojection operator onto $\ZS_k(\Delta)$. Denote $|\Delta|=\max_j(t_{j+1}-t_j)$. Given sequence of fixed $k$-order knot-collections $\Delta_n$ with $|\Delta_n|\to 0$ consider the operator sequence $P_{\Delta_n}$. It is known such an operator sequence satisfy properties \p{C1}-\p{C4}. Moreover, Shadrin \cite{Sha} solving de Boor's conjecture proved the uniformly boundedness of these operators on $C[a,b]$ that implies $P_{\Delta_n}(f)\rightrightarrows f$ whenever $f\in C[a,b]$.  Then Passenbrunner and  Shadrin in \cite{PaSh} proved $P_{\Delta_n}(f)\to f$ a.e. for any $f\in L^1[a,b]$, where it was also proved that the kernels of operators $P_{\Delta_n}$ form an approximation of the identity. Thus we can state the following.
\begin{corollary}\label{C5}
		Let $\Delta_n$ be a sequence of knot-collections in $[a,b]$ such that $|\Delta_n|\to0$. Then for a set $E\subset [a,b]$ to be the divergence (unbounded divergence) set of $P_{\Delta_n}(f)$ for a function $L^p[a,b]$, $1\le p\le \infty$ it is necessary and sufficient to be a $G_{\delta\sigma}$-set ($G_\delta$-set) in the standard topology of $[a,b]$.
\end{corollary}
\subsection{More specific examples}
Let us also discuss some more specific examples of operator sequences that obey the conditions of Theorem \ref{T1} and \ref{T2}, and those can be deduced from one of above considered general examples. 
\begin{enumerate}
	\item Partial sums of Fourier series in Haar and Franklin systems with general nodes are spline operator sequences, so those obey the conditions of \cor{C5}. For the Haar series the characterization of $C$, $D$ and $UD$ sets first were given in \cite{Lun} (for $UD$ sets) and \cite{Kar3} (for $D$ sets)
	\item The characterization of convergence and divergence sets of $(C,\alpha)$-means of trigonometric Fourier series with a parameter $\alpha>0$ was given by Zahorski in \cite{Zah1}, where the same problem was considered also for Poisson integrals on the unit disc. We just note that both operators are approximations of the identity on the unit circle.
	\item Let $(X,d,\mu)$ be a space of homogeneous type such that $C_K(X)$ is dense in $L^1(X)$. Consider a sequence of measurable sets $G_n\subset X$, which are equivalent to balls $B(0,r_n)$, i.e. 
	\begin{equation}
		G_n\subset B(0,r_n),\quad \mu(G_n)>c\mu(B(0,r_n))
	\end{equation}
where $c>0$ is a constant. Clearly the operator sequence 
\begin{equation}
	U_n(x,f)=\frac{1}{\mu(G_n)}\int_{G_n}f(t)d\mu(t)
\end{equation}
is a particular case of operators \e{x61}. This sequence in particular is interesting when $X$ coincides with $\ZR^n$.
\item Dyachkov in \cite{Dya} considered the $C$-$D$ sets characterization problem for operators
\begin{equation*}
	U_\varepsilon(x,f)=\int_{\ZR^d}f(x-t)\phi_\varepsilon(t)dt\hbox{ as }\varepsilon\searrow 0,
\end{equation*}
for an approximation of the identity kernels $\phi_\varepsilon(t)=\varepsilon^{-d}\phi(t/\varepsilon)$, where $\phi\in L^1(\ZR^d)$.
\item The main part of Zahorski's theorem of \cite{Zah} (\trm{OT3}) is a construction of a continuous function whose non-differentiability set is a given set $G=G_1\cup G_2$, where $G_1$ is $G_\delta$ set and $G_2$ is a $G_{\delta\sigma}$ null set. Without loss of generality one can suppose that $G_1\cap G_2=\varnothing$. So the construction can be split into contractions of two different functions $f_1$ and $f_2$, which non-differentiability sets are $G_1$ and $G_2$ respectively. Then the desired functions is $f=f_1+f_2$. The function $f_2$ can be seeked in the form
\begin{equation}
	f_2(x)=\int_{-\infty}^xg(t)dt
\end{equation}
where $g\in L^\infty(\ZR)$. Then the differentiability points set of the function $f_2$ coincides with the set of points $x\in \ZR$ such that the limit
\begin{equation}\label{x65}
	\lim_{|I|\to 0,\,I\ni x}\frac{1}{|I|}\int_Ig(t)dt 
\end{equation}
exists, where $I$ are closed intervals containing the point $x$. The existence of a function $g\in L^\infty(\ZR)$ for which the limit \e{x65} exists only at the points of a given $G_{\delta\sigma}$ null set, is a part of general \pro{P}.
\end{enumerate}

	\section{Auxiliary lemmas}

	It is known from \cite{MaSe} that for any quasi-distance $d$ there exists an alternative quasi-distance $d'$, which is equivalent to $d$, i.e. $C_1d(x,y)\le d'(x,y)\le C_2d(x,y)$, and satisfies a Lipschits type condition 
	\begin{equation}\label{x63}
		|d'(x,z)-d'(y,z)|\le L(d'(x,y))^\alpha(\max\{d'(x,z),d'(y,z)\})^{1-\alpha},
	\end{equation}
	where $L>0$ and $0<\alpha<1$ are some constants. One can see that $d'$ induces the same topology as $d$ does and in view of \e{x63} $d'$-balls become open sets there. 
	Using this we can prove the following.
	\begin{lemma}\label{L13}
		Any open set $G$ in a space of homogeneous type permits a representation $G=\cup_{k=1}^\infty F_k$, where each $F_k$ is a closed set and $F_k\Subset G$.
	\end{lemma}
	\begin{proof}
	Let $\dist'$ denote the distance corresponding to the quasi-distance $d'$ from \e{x63}. Define
		\begin{equation*}
			F_k=\{x\in G:\, \dist'(x,G^c)\ge 1/k\},\quad k=1,2,\ldots.
		\end{equation*}
	It is clear that $F_k\Subset G$. It remains to show that each $F_k$ is closed or equivalently $F_k^c$ is open. For any $x\in F_k^c$ we have $\dist'(x,G^c)<1/k$. Then for small enough $r>0$ the condition $d'(y,x)<r$ implies $\dist'(y,G^c)<1/k$, and it follows from \e{x63}. Thus we obtain $B'(x,r)\subset F_k^c$, i.e. $F_k^c$ is open.
	
	\end{proof}
The following lemma is a standard property of a topology induced by a quasi-distance.
\begin{lemma}\label{L2}
	If sets $A$ and $B$ in a quasi-metric space $(X,d)$ satisfy $\dist(A,B)>0$, then there are disjoint open sets $U$ and $V$ such that $A\Subset U$ and $B\Subset V$.
\end{lemma}
\begin{proof}
	Let $d'$ be a Lipschits distance satisfying \e{x63}. Equivalency of distances $d$ and $d'$ implies $\dist'(A,B)>0$, besides we know that each $d'$-ball is an open set. Define an open set $U\supset A$ to be the union of all the balls $B'(x,r)$ with $x\in A$ and $r=\dist'(x,B)/(2K)>0$, where $K$ is the triangle inequality constant of the distance $d'$. Similarly defining the open neighborhood $V\supset B$, we claim that $U\cap V=\varnothing$. Suppose to the contrary that there is a point $z\in U\cap V$. Then we find $x\in A$ and $y\in B$ such that 
	\begin{equation*}
		d'(x,z)<\dist'(x,B)/(2K), \quad d'(y,z)<\dist'(y,A)/(2K).
	\end{equation*}
	This implies
	\begin{equation*}
		d'(x,y)\le K(d'(x,z)+d'(y,z))<\max\{\dist'(x,B),\dist'(y,A)\},
	\end{equation*}
	which is a contradiction and so $U\cap V=\varnothing$. It remains just observe that $A\Subset U$ and $A\Subset V$, which hold with respect to both distances $d$ and $d'$.
\end{proof}
		It is known in the geometric measure theory (see for example \cite{Sim}, Theorem 1.11) that if in a topological measurable space $(X,\mu)$ every open set is of $F_\sigma$ and the measure $\mu$ is open $\sigma$-finite (i.e. $X$ is a countable union of finite measure open sets), then 
		\begin{equation}\label{x67}
			\mu(E)=\inf_{G\text{ open},\,  G\supset E}\mu(G),
		\end{equation}
	for any Borel set $E$.  Any space of homogeneous type $X$ satisfies such properties, since every open set is of $F_\sigma$ by \lem{L13}, and $X$ can be written as a countable union $d'$-balls, which are open and have finite measure. Hence we can state the following.
	\begin{lemma}\label{L10}
	Any space of homogeneous type $X$ enjoys the property \e{x67}.
	\end{lemma}
\begin{definition}
	A sequence $\Omega=\{\omega_k:\,k=1,2,\ldots\}$ of Borel sets in a topological space $X$ is said to be
	
	1) a {\it partition} of an open set $B\subset X$ if
	\begin{equation}\label{x42}
		B=\cup_k\omega_k,\quad \omega_k\cap \omega_{k'}=\varnothing \text { if }k\neq k',
	\end{equation}
	
	2) {\it locally-finite} if there are open sets $V_k\supset \omega_k$ such that each of them has  only finite number of intersection with others, i.e.
	\begin{equation}\label{x44}
		\#\{j\in \ZN:\,V_j\cap V_k\neq \varnothing\}<\infty\text{ for any }k=1,2,\ldots,
	\end{equation}
	
	3) a {\it regular partition}  for an open set $B$ if it is a partition of $B$ and there are open sets $V_j$, satisfying \e{x44} (so it is locally-finite) as well as,
	\begin{equation}\label{x43}
	 \omega_k\Subset V_k\subset B,\quad k\ge 1.
	\end{equation}

\end{definition}
\begin{lemma}\label{L11}
	Let $\omega_k$, $k=1,2,\ldots$, be a locally-finite family of measurable sets in a topological measurable space $X$. Then for any measurable set $E\subset \cup_k\omega_k$ and numbers $\varepsilon_k >0$ there exists an open set $G\supset E$ such that 
	\begin{equation*}
		\mu((G\setminus E)\cap \omega_k)<\varepsilon_k,\quad k=1,2,\ldots.
	\end{equation*}
\end{lemma}
\begin{proof}
	Let  $V_k\supset \omega_k$, $k=1,2,\ldots$, be open sets such that each $V_k$ has intersection only with finite number of elements $V_j$ and we denote this number by $l_k$. One can check that 
	\begin{equation}
		\varepsilon'_k=\min_{j:\,V_j\cap V_k\neq \varnothing}\varepsilon_j/l_j>0\text{ for all }k=1,2,\ldots,
	\end{equation}
	because only finite number of $j$'s in this minimum satisfy the relation $V_j\cap V_k\neq \varnothing$. Thus, applying \lem{L10}, we may define open sets $G_k$, $k=1,2,\ldots$, such that
	\begin{align}
		E\cap \omega_k\subset G_k\subset V_k,\quad \mu(G_k\setminus E)<\varepsilon'_k.
	\end{align}
	Thus we have $G=\cup_kG_k\supset E$ and 
	\begin{align}
		\mu((G\setminus E)\cap \omega_k)&\le \sum_{j:\,V_j\cap \omega_k\neq \varnothing} \mu((G_j\setminus E )\cap \omega_k)\\
		&\le  \sum_{j:\,V_j\cap V_k\neq \varnothing} \mu(G_j\setminus E)<\sum_{j:\,V_j\cap V_k\neq \varnothing} \varepsilon'_j\\
		&<  \sum_{j:\,V_j\cap V_k\neq \varnothing} \varepsilon_k/l_k=\varepsilon_k.
	\end{align}
This completes the proof.
\end{proof}

	\begin{lemma}\label{L0}
		Any open set $B\subset X$ in a space of homogeneous type $X$ enjoys a regular partition.
\end{lemma}
\begin{proof}
	For any open set $B\subset X$ we can find closed sets $F_k$, $k=1,2,\ldots$,  such that 
	\begin{equation}\label{x41}
		B=\cup_kF_k,\quad F_k\Subset \Int(F_{k+1}), 
	\end{equation}
where $\Int(E)$ denotes the interior of a set $E\subset X$. Indeed, first we write $B=\cup_k A_k$, where $A_k\Subset B$ are closed sets (see \lem {L13}). Since $\dist(A_1,B^c)>0$, by \lem{L2} the sets $A_1$ and $B^c$ have disjoint open neighborhoods $D_1\Supset A_1$ and $G_1\Supset B^c$. So we will have $A_1\Subset D_1$ and $\overline {D_1}\subset (G_1)^c\Subset B$. Then we do the same with the closed sets $A_2\cup \overline {D_1}\Subset B$ and $B^c$ and we get an open set $D_2\Supset  A_2\cup \overline {D_1}$ such that $\overline {D_2}\Subset B$. Continuing this procedure to infinity, finally we get a sequence of open sets $D_k$ such that $D_k\Supset A_k\cup \overline{D_{k-1}}$ and $\overline {V_k}\subset B$. Let us see that $F_1=A_1$, $F_k=A_k\cup \overline{D_{k-1}}$, $k\ge2$, are closed sets, satisfying \e{x41}. Clearly we have $B=\cup_kF_k$, besides $F_k\Subset D_k\subset F_{k+1}$. Since $D_k$ is open, we get $F_k\Subset \Int(F_{k+1})$.
	
	Having \e{x41}, we define Borel sets $\omega_1=F_1$, $\omega_k=F_k\setminus F_{k-1}$, $k\ge 2$, which clearly satisfy the partition condition \e{x42}. One can also see that the open sets $V_1=\Int(F_2), V_2=\Int(F_3)$, $V_k=  \Int(F_{k+1})\setminus F_{k-2}$, $k>2$, satisfy \e{x44}. Besides by \e{x41} we have
	\begin{equation*}
		\omega_k\subset F_k\setminus \Subset \Int(F_{k+1})\subset \Int(F_{k-1})\subset V_k, 
	\end{equation*}
	which implies \e{x43}.

\end{proof}
 The notation $f_n(x)\rightrightarrows f(x)$ will stand for the uniformly convergence of sequences of functions on a set.
	\begin{lemma}\label{L1}
		Let $(X,d,\mu)$ be a space of homogeneous type, and let an operator sequence \e{1} satisfy conditions \p{C1} and \p{C3}. If $G\subset X$ is an open set of finite measure and $F\Subset G$ is measurable, then 
		\begin{equation}\label{5}
			d_G(F)=\sup_{n\in\ZN}\sup_{x\in G^c}\,\, \sup_{\|f\|_1\le 1}|U_n(x,f\cdot \ZI_F)|<\infty.
		\end{equation}
	\end{lemma}
	\begin{proof}
First we note that by property \p{C3} we have
	\begin{equation}\label{x56}
		U_n(x,f\cdot \ZI_F)\rightrightarrows 0\hbox{ on } G^c,\hbox{ as }n\to\infty,
	\end{equation}
for any $f\in L^1(X)$, since $f\cdot \ZI_F\in L^1(X)$ is identically zero on the on the open set $F^c$ and $G^c\Subset F^c$. Suppose to the contrary that $d_G(F)=\infty$.
		The boundedness property \p{C1} of operators \e{1} implies
		\begin{align}
			d_{G,m}(F)=\max_{1\le n\le m}\sup_{x\in G^c}&\sup_{\|f\|_1\le 1}|U_n(x,f\cdot \ZI_F)|\\
			&\le \max_{1\le n\le m}\rho_n<\infty,\text{ for any }m\in \ZN,\label{a11}
		\end{align}
		and by the contrary assumption we have
		\begin{equation}\label{a12}
			d_{G,m}(F)\to\infty,\hbox{ as } m\to\infty.
		\end{equation}
		Applying induction, one can find integers $n_k\in\ZN$, points $x_k\in G^c$ and functions $f_k\in L^1(X)$, $k=1,2,\ldots $, such that
		\begin{align}
			&\| f_k\|_1\le 1, \quad \supp f_k\subset F,\, k\ge 1,\label{a5}\\
			&|U_{n_k}(x_k,f_k)|>k^3\cdot(1+\max_{1\le i<k}\rho_{n_i}),\, k\ge 1\label{a6}\\
			&\sup_{1\le i<k}|U_{n_k}(x,f_i)|<1,\quad x\in G^c,\, k>1.\label{a7}
		\end{align}
		Indeed, using \e{a12}, first we define $f_1(x)$, $n_1\in\ZN$ and $x_1\in G^c$ satisfying
		\begin{align*}
			&\| f_1\|_1\le 1, \quad \supp f_1\subset F,\\
			&|U_{n_1}(x_1,f_1)|>1,
		\end{align*}
		which is the base of induction. Then suppose that we have already defined $f_k\in L^1(X)$, $n_k\in\ZN$ and $x_k\in G^c$ satisfying conditions \e{a5}-\e{a7} for $k=1,2,\ldots ,p$. Once again using \e{a12}, we find $f_{p+1}(x)$, $n_{p+1}\in\ZN$ and $x_{p+1}\in G^c$, satisfying \e{a5} and \e{a6} for $k=p+1$. Moreover, by \e{a11}  the number $n_{p+1}$ can be chosen enough bigger and so, using \e{x56}, we can also ensure \e{a7} for $k=p+1$. This completes the induction process. Let us also add the bound
		\begin{equation}\label{x45}
			|U_{n_k}(x,f_i)|\le \rho_{n_k},\quad x\in X,\quad i=1,2,\ldots,
		\end{equation}
		which immediately follows from \p{C1} and the bound $\|f_i\|_1\le 1$. Now consider the function
		\begin{equation}\label{a15}
			f(x)=\sum_{k=1}^\infty \alpha_kf_k(x),\quad \alpha_k=\frac{1}{k^2\cdot(1+\max_{1\le i<k}\rho_{n_i})}.
		\end{equation}
		Clearly $f\in L^1 (X)$ and moreover, $\supp f \subset F$. Using \e{a5},\e{a6},\e{a7},\e{x45} and \e{a15}, we obtain
		\begin{multline*}\begin{split}
				|U_{n_k}(x_k,f)|\ge \\
				&\ge \alpha_k|U_{n_k}(x_k,f_k)|-\sum_{i=1}^{k-1}\alpha_i|U_{n_k}(x_k,f_i)|-
				\sum_{i=k+1}^\infty \alpha_i |U_{n_k}(x_k,f_i)|\\
				&\ge k-\sum_{i=1}^{k-1}\frac{1}{i^2}-\sum_{i=k+1}^\infty \frac{1}{i^2}\ge k-2.
			\end{split}
		\end{multline*}
		This gives a contradiction, since $\supp f \subset F$ and $U_n(x,f)\rightrightarrows 0$ on $G^c$ according to  \e{x56}.
	\end{proof}

For a sequence of pairwise disjoint measurable sets $\Omega=\{\omega_k\}$ in a space of homogeneous type $X$ and a mapping  $\nu:\Omega\to \ZN$ we denote
\begin{equation}\label{36}
	U_{(\Omega,\nu)}(x,f)=\sum_{k\ge 1} U_{\nu(\omega_k)}(x,f)\cdot \ZI_{\omega_k}(x).
\end{equation}
\begin{lemma}\label{L7}
	Let an operator sequence $U_n$ satisfy \p{C3} and functions $f_k\in L^1(X)$, $k=1,2,\ldots,n$ be constant on an open set  $G\subset X$, which has a partition $\Omega=\{\omega_k\}$ such that $ \omega_k\Subset G$. Then for any $\varepsilon >0$ and $R>0$ there exists a mapping $\nu:\Omega\to \ZN$ such that
	\begin{align}
		&\min_{\omega\in \Omega}\nu(\omega)>R,\label{x40}\\
		&|U_{(\Omega,\nu )} (x,f_i)-f_i(x)|<\varepsilon, \quad x\in G, \quad i=1,2,\ldots ,n.\label{x23}
	\end{align}
\end{lemma}
\begin{proof} From condition \p{C3} it follows that $U_n(x,f_i)\rightrightarrows f_i(x)$ on each set $\omega_k$, for any $i=1,2,\ldots n$. Thus for every $\omega_k$ we may find an integer $\nu(\omega_k)\in \ZN$ such that $\nu(\omega_k)>R$ and
	\begin{equation*}
		|U_{\nu(\omega_k)} (x,f_i)-f_i(x)|<\varepsilon, \quad x\in \omega_k, \quad i=1,2,\ldots ,n.
	\end{equation*}
Thus we get \e{x40} and \e{x23}.
\end{proof}

\begin{lemma}\label{L12}
	Let an operator sequence $U_n$ satisfy \p{C1} and \p{C3}, $G\subset X$ be an open set with a regular partition $\Omega=\{\omega_k\}$ and let $\nu:\Omega\to \ZN$ be a mapping. Then if $C\subset G$ is a measurable set and a function $f\in L^\infty(X)$ satisfies
	\begin{equation}\label{x32}
		\|f\|_\infty\le 1,\quad \supp (f)\subset C,
	\end{equation}
	then it holds the inequality
	\begin{equation}\label{k5}
		|U_{(\Omega,\nu)}(x,f)|\le \sum_{j}c(\omega_j)|C\cap \omega_j|,\quad x\in X,
	\end{equation}
	where $c(\omega_j)>0$ are constants independent of $C$ and the function $f$. In fact, those depend only on the operators $U_n$, the open set $G$ and its partition $\omega_k$ (see \e{x46}).
\end{lemma}
\begin{proof}  For $x\not\in G$  \e{k5} trivially holds, since we have $U_{(\Omega,\nu)}(x,f)=0$ (see \e{36}). By the definition of regular partition we may fix open sets $V_k\supset \omega_k$, $k=1,2,\ldots$, satisfying \e{x42}. For $x\in G$ we can write
	\begin{equation}\label{x27}
		x\in \omega_k\text{ for a unique }k\ge1,
	\end{equation}
	and then
	\begin{align}
		|U_{(\Omega,\nu)}(x,f)|\le \sum_{j:\,V_j\cap V_k=\varnothing} &|U_{\nu(\omega_k)}(x,f\cdot\ZI_{\omega_j})|\\
		&+\sum_{j:\,V_j\cap V_k\neq \varnothing} |U_{\nu(\omega_k)}(x,f\cdot\ZI_{\omega_j})|=S_1(x)+S_2(x).\label{x34}
	\end{align}
	The condition $V_j\cap V_k=\varnothing$ implies $x\in\omega_k\Subset V_k\subset (V_j)^c$. So applying \lem{L1} and relations \e{x32}, for the first sum we obtain
	\begin{align}
		S_1	(x)\le \sum_{j:\,V_j\cap V_k=\varnothing}&\|f\cdot\ZI_{\omega_j}\|_{1}d_{V_j}(\omega_j)\\
		&\le \sum_{j:\,V_j\cap V_k=\varnothing}d_{V_j}(\omega_j)|C\cap \omega_j|.\label{x31}
	\end{align}
	Applying \p{C1} and \e{x32}, we have
	\begin{align}
		S_2(x)&\le \sum_{j:,V_j\cap V_k\neq \varnothing}\rho_{\nu(\omega_k)}|C\cap \omega_j|\\
		&\le \sum_{j:,V_j\cap V_k\neq \varnothing}\max\{\rho_{\nu(\omega_i)}:\,V_i\cap V_j\neq \varnothing \}|C\cap \omega_j|.\label{k4}
	\end{align}
	Combining \e{x34}, \e{x31} and \e{k4} we get \e{k5} with the constants
	\begin{equation}\label{x46}
		c(\omega_j)=\max\left\{\max\{\rho_{\nu(\omega_i)}:\,V_i\cap V_j\neq \varnothing \},d_{V_j}(\omega_j)\right\}.
	\end{equation}
	It remains just note that by \p{C1} and \lem{L1} we have $c(\omega_j)<\infty$ for any $j=1,2,\ldots$. 
\end{proof}
For an open set $A\subset X$ of bounded measure we denote
	\begin{equation*}
		\lambda(A)=A\cup \left\{x\in X: \limsup_{n\to \infty}|U_n(x,\ZI_A)| >0\right\}.
	\end{equation*}
	If an operator sequence $\{U_n\}$ satisfies \p{C4}, then
	\begin{equation*}
		\lim_{n\to \infty}U_n(x,\ZI_A)=0,\hbox { a.e. on } X\setminus A,
	\end{equation*}
	so we obtain
	\begin{equation}\label{lam}
		\mu(\lambda(A) \setminus A)=0.
	\end{equation}

	\begin{lemma}\label{L3}
		Let an operator sequence \e{1} satisfy \p{C1}, \p{C3} and \p{C4}. If open sets $A$ and $B$ in $X$ have finite measures and $\lambda(A)\subset B$,  then for any $\varepsilon>0$ there exists an open set  $G\subset B$ such that
		\begin{align}
			&\lambda(A)\subset G,\quad \lambda(G)\subset B,\label{b1}\\
			&\mu(G\setminus A)<\varepsilon.
		\end{align}
	\end{lemma}
	\begin{proof}
		Applying \lem{L0}, we find for the open set $B$ a regular partition $\Omega=\{\omega_k:\, k=1,2,\ldots\}$. By \lem{L1} and relation \e{x43} we conclude $d_B(\omega_k)<\infty$ for any $k=1,2,\ldots$. Applying \lem{L11}, we find an open set $G\subset B$ such that
		\begin{align}
			&\lambda(A)\subset G,\label{b3}\\ 
		   &\mu((G\setminus A)\cap \omega_k)<\varepsilon_k=\frac{\varepsilon}{2^k(1+d_B(\omega_k))}.\label{x18}
		\end{align}
The first relation in \e{b1} immediately follows from \e{b3}. Then we have 
\begin{equation}
\mu(G\setminus A)\le \sum_{k=1}^\infty \mu((G\setminus A)\cap \omega_k)<\sum_{k=1}^\infty \frac{\varepsilon}{2^k}=\varepsilon.
\end{equation}
 It remains to prove that $\lambda(G)\subset B$. Chose an arbitrary point $x\not\in B$. The relation $\lambda(A)\subset B$ implies 
\begin{equation}
	\lim_{n\to \infty}U_n(x,\ZI_A)=0.
\end{equation}
Thus, by the linearity of operator $U_n$ and the pairwise disjointness of the sets $\omega_k$ 
we can write
		\begin{align}
			\limsup_{n\to\infty }|U_n(x,\ZI_G)|&=\limsup_{n\to\infty }|U_n(x,\ZI_{G\setminus A})|\\
			&\le \limsup_{n\to\infty }\sum_{k=0}^\infty |U_n(x,\ZI_{(G\setminus A)\cap \omega_k})|\label{b9}
		\end{align}
	By property \p{C3} we have
	\begin{equation}\label{t1}
		\lim_{n\to\infty}U_n(x,\ZI_{(G\setminus A)\cap \omega_k})=0 \text{ for every }k\ge 1.
	\end{equation}
Besides by \lem{L1} and \e{x18} we can write
\begin{align}
|U_n(x,\ZI_{(G\setminus A)\cap \omega_k})|\le d_B(\omega_k)\|\ZI_{(G\setminus A)\cap \omega_k}\|_1< \varepsilon_k d_B( \omega_k)\le  \varepsilon/2^k.\label{t2}
\end{align}
Combining \e{b9}, \e{t1} and  \e{t2}, we obtain
		\begin{equation*}
			\limsup_{n\to\infty }U_n(x,\ZI_{G})=0,
		\end{equation*}
i.e. $x\not\in \lambda(G)$ and therefore $\lambda(G)\subset B$.
	\end{proof}
\begin{lemma}\label{L4}
	Let an operator sequence \e{1} satisfy \p{C1}, \p{C3} and \p{C4}. If $A$ and $B$ are open sets of finite measure in $(X,\mu)$ with $\lambda(A)\subset B$, then there exists a family of open sets
	\begin{equation*}
		G_r,\quad r\in \ZD=\left\{i/2^k,\, 0\le i\le 2^k,\, k=0,1,\ldots \right\}
	\end{equation*}
such that
	\begin{align}
		&G_1=A,\,G_0=B,\label{d1}\\
		& \lambda(G_r)\subset G_{r'}\text{ whenever } r>r'.\label{x47}
	\end{align}
\end{lemma}
\begin{proof}
	Define $G_1=A$, $G_0=B$ and apply \lem{L3} to the pair of open sets $G_1,\,G_0$. With this we find an open set $G=G_{1/2}$, satisfying \e{b1}. Then we apply induction. Denote
	\begin{equation*}
		\ZD_k=\left\{\frac{i}{2^k}:\, 0\le i\le 2^k\right\},
	\end{equation*}
	and suppose that we have already defined $G_r$ for each $r\in\ZD_k$ such that \e{x47} holds whenever $r,r'\in \ZD_k$. Applying \lem{L3} to each pair of sets $G_{i/2^k},\,G_{(i+1)/2^k}$, we obtain intermediate sets $G_{(2i+1)/2^{k+1}}$, $0\le i\le 2^k-1$. Obviously a new family of sets $\{G_r,\, r\in\ZD_{k+1}\}$ will also satisfy \e{x47}. 
	Continuing the induction procedure we finally get sets $G_r$ defined for all
	$r\in\ZD$, and satisfying \e{x47} for full range of dyadic indexes $r,r'$. 
\end{proof}
\begin{lemma}\label{L6}
Let operators \e{1} satisfy \p{C1}-\p{C4}. If $\varepsilon >0$, $G\subset X$ is a finite measure open set and $E\subset G$ is a null-set, then there exists an open set $A$, with $E\subset A \subset G$, and a function $h(x)$, $x\in X$, such that
	\begin{align}
		&\mu(A)<\varepsilon,\label{x38}\\
		&\supp (h)\subset G,\quad h(x)=1,\,x\in A,\label{f1}\\
		& 0\le h(x)\le 1,\quad x\in X,\label{f2}\\
		&|U_n(x,h)|\le \varepsilon ,\quad x\in G^c,\quad n=1,2,\ldots, \label{f3}\\
		&U_n(x,h)\to h(x) \hbox{ for every } x\in X.\label{f4}
	\end{align}
\end{lemma}
\begin{proof}
Applying \lem{L0}, we find a regular partition $\Omega=\{\omega_k:\,k\in \ZN \}$ for $G$.
Then, using  \lem{L11} we define an open set $B$, satisfying
	\begin{align}
		&E\subset B\subset G,\label{x19}\\ 
		&\mu(B\cap \omega_k)<\varepsilon_k=\frac{\varepsilon}{2^k(1+d_G(\omega_k))}.\label{x20}
	\end{align} 
 Applying \lem{L3}, we get an open set $A$, satisfying  $E\subset A$, $\lambda(A)\subset B$. Bound \e{x20} implies $\mu(B)<\varepsilon $ and so \e{x38}. 
	According to \lem{L4}, there exists a family of open sets $\{G_r:\,r\in \ZD\}$, satisfying \e{d1} and \e{x47}.  Define a function
	\begin{equation}\label{c12}
		h(x)=\left\{
		\begin{array}{ccl}
			\sup\{r:\, x\in G_r\} &\hbox{ if }& x\in B,\\
			0 &\hbox{ if }& x\in X\setminus B.
		\end{array}
		\right.
	\end{equation}
	Obviously, $h(x)$ satisfies conditions \e{f1}, \e{f2} and moreover, $\supp (h) \subset B$. Using \lem{L1} and \e{x20}, for any $x\in G^c$  we obtain
	\begin{align}
		|U_n(x,h)|&\le \sum_{k=1}^\infty |U_n(x, h\cdot \ZI_{\omega_k})|\\
		&<\sum_{k=1}^\infty d_G(\omega_k)|B\cap \omega_k|\le \sum_{k=1}^\infty\frac{\varepsilon}{2^k}=\varepsilon,
	\end{align}
	that gives \e{f3}. It remain to check condition \e{f4}. To this end we consider the function
	\begin{equation}\label{c4}
		p(x)=\ZI_{G_{r_0}}(x)+\sum_{k=0}^{m-1}r_k\ZI_{G_{r_{k+1}}\setminus G_{r_k}}(x),
	\end{equation}
	where the dyadic rational numbers $r_k\in \ZD$ satisfy 
	\begin{equation}\label{x21}
		1=r_0>r_1>\ldots>r_m=0,\quad r_{k}-r_{k+1}<\varepsilon.
	\end{equation}
Observe that
	\begin{equation}\label{c11}
		|h(x)-p(x)|<\varepsilon,\quad x\in X.
	\end{equation}
Indeed, we have
	\begin{align}
		&h(x)=p(x)=1,\quad x\in G_1 =A,\label{c18}\\
		&h(x)=p(x)=0,\quad x\in X\setminus G_0=X\setminus B.\label{c19}
	\end{align}
	If $x\in G_0\setminus G_1$, then we have $x\in G_{r_{k+1}}\setminus G_{r_k}$ for some $k=0,1,\ldots ,m-1$.
	Then from the definition of $h(x)$ it follows that 
	\begin{equation}\label{x39}
		r_k\ge h(x)\ge r_{k+1},\quad p(x)=r_k.
	\end{equation}
	Thus we get \e{c11}. From \e{c11} and the bound $\|U_n\|_{L^\infty\to M}\le \varrho $ (see property \p{C2}) it follows that
	\begin{multline}\label{c10}
		\limsup_{n\to\infty}|U_n(x,h)-h(x)| \le \limsup_{n\to\infty}|U_n(x,h-p)-h(x)+p(x)|\\
		+\limsup_{n\to\infty}|U_n(x,p)-p(x)|
		\le (\varrho+1)\varepsilon +\limsup_{n\to\infty}|U_n(x,p)-p(x)|.
	\end{multline}
	Consider the following cases.
	
	\noindent
	\underline{\it Case 1}: $x\in G_0=G_{r_0}$. We have $h(t)=1$ as $t\in G_0$.  By \lem{L13} there exists a closed set $F$ such that $x\in F\Subset G_0$. Thus, from property \p{C3} we obtain
	\begin{equation*}
		\lim_{n\to\infty}U_n(x,h)=1=h(x).
	\end{equation*}
	
	\noindent
	\underline{\it Case 2}: $x\in B\setminus G_0$. In this case we have
	\begin{equation}\label{c13}
		x\in G_{r_{k+1}}\setminus G_{r_k}
	\end{equation}
	for some $k=0,1,\ldots ,m-1$. Again by \p{C3} it follows that
	\begin{equation}\label{x24}
		\lim_{n\to\infty}U_n(x,\ZI_{G_{r_i}})=1,\quad i\ge k+1.
	\end{equation}
	On the other hand, using the property $\lambda(G_{r_i})\subset G_{r_{i+1}}$, we can write
	\begin{equation}\label{x25}
		\lim_{n\to\infty}U_n(x,\ZI_{G_{r_i}})=0,\quad i\le k-1,
	\end{equation}
	whenever $k\ge 1$. Thus we obtain
	\begin{multline}\label{c9}
		\lim_{n\to\infty}U_n(x,\ZI_{G_{r_{i+1}}\setminus G_{r_i}}) =\lim_{n\to\infty}U_n(x,\ZI_{G_{r_{i+1}}})\\- \lim_{n\to\infty}U_n(x,\ZI_{G_{r_i}})=0,\, i\in \ZN,\,i\neq k,k-1.
	\end{multline}
Thus, first using \e{x39}, \e{c9} then \e{x24}, \e{x25} and the boundedness property \p{C2}, in the case $k>0$ we get
\begin{align}
		\limsup_{n\to\infty}&|U_n(x,p)-p(x)|\\
		&=\limsup_{n\to\infty}|r_{k-1}U_n(x,\ZI_{G_{r_{k}}\setminus G_{r_{k-1}}})+r_kU_n(x,\ZI_{G_{r_{k+1}}\setminus G_{r_k}})-r_k|\\
		&\le \varrho \varepsilon+\limsup_{n\to\infty}|r_{k}U_n(x,\ZI_{G_{r_{k}}\setminus G_{r_{k-1}}})+r_kU_n(x,\ZI_{G_{r_{k+1}}\setminus G_{r_k}})-r_k|\\
		&= \varrho \varepsilon+r_k\limsup_{n\to\infty}|U_n(x,\ZI_{G_{r_{k+1}}}) -U_n(x,\ZI_{G_{r_{k-1}}}) -1|\\
		&= \varrho \varepsilon.
\end{align}
If $k=0$, then similarly we have
\begin{equation}
	\limsup_{n\to\infty}|U_n(x,p)-p(x)|\le  \varrho \varepsilon+r_0\limsup_{n\to\infty}|U_n(x,\ZI_{G_{r_1}})  -1|= \varrho \varepsilon.
\end{equation}
	Combining the last two estimates with \e{c10}, we conclude $U_n(x,h)\to h(x)$.
	
	\noindent
	\underline{\it Case 3}: $x\in X\setminus B$. From the relations $\lambda(G_{r_i})\subset G_{r_m}=B$, $i=1,2,\ldots ,m-1$, we have
	\begin{equation*}
		\lim_{n\to\infty}U_n(x,\ZI_{G_{r_{i+1}}\setminus G_{r_i}})=0,\quad i=1,2,\ldots ,m-2,
	\end{equation*}
	and therefore using also \e{c4} and the equation $r_{m-1}<\varepsilon$ (see \e{x21}), we obtain
	\begin{equation*}
		\limsup_{n\to\infty}|U_n(x,p)|=\limsup_{n\to\infty}r_{m-1}|U_n(x,\ZI_{G_{r_m}\setminus G_{r_{m-1}}})|<\varepsilon \varrho,
	\end{equation*}
	which completes the proof of \e{f4}.

\end{proof}

\begin{lemma}\label{L8}
If operators \e{1} satisfy \p{C1}-\p{C4}, then for every $G_\delta$ null-set $E\subset X$ and $\varepsilon>0$, there exists a function $g\in M(X)$, satisfying
\begin{align}
	&\mu(\supp (g))<\varepsilon,\label{g4}\\
		&0\le g(x)\le 1,\label{g1}\\
	    &U_n(x,g)\to g(x),\quad  x\in X\setminus E,	\label{g2}\\
	   & \limsup_{n\to\infty}U_n(x,g)\ge 1,\quad \liminf_{n\to\infty}U_n(x,g)\le 0, \quad x\in E.\label{g0}
\end{align}
\end{lemma}
\begin{proof}
We have $E=\cap_{k=1}^\infty E_k$, where $E_k\subset X$ are open sets and we can also suppose that each 
\begin{equation}\label{x68}
	\mu(E_k)<\varepsilon/2^k.
\end{equation} 
We claim there are functions $h_k(x)$, $x\in X$, and open sets $A_k\subset X$, $k=1,2,\ldots,$ satisfying the following conditions h1)-h8):

\vspace{2mm}
h1) $E\subset A_k\subset E_k$, $A_k\subset A_{k-1}$, $k\ge 1$ ($A_0=X$),

\vspace{2mm}
h2) $\supp h_k(x)\subset A_{k-1}$, $h_k(x)=1$, $x\in A_k$, $k\ge 1$,

\vspace{2mm}
h3) $0\le h_k(x)\le 1$, $x\in X$, $k\ge 1$,

\vspace{2mm}
h4) $|U_n(x,h_k)|<2^{-k}$, $x\in A_{k-1}^c$, $k\ge 1$,

\vspace{2mm}
h5) $U_n(x,h_k)\to h_k(x)$ as $n\to\infty $, $x\in X$.

\vspace{2mm}
\noindent
For any open set $A_k$ there is a regular partition $\Omega_k$ and a function $\nu_k:\Omega_k\to\ZN$, satisfying 

\vspace{2mm}
h6) $\min_{\omega\in\Omega_k}\nu_k(\omega)>k$ for any $k=1,2,\ldots$.

\vspace{2mm}
h7) $|U_{(\Omega_k,\nu_k)}(x,h_i)-1|<\frac{1}{k^2}$, $\quad x\in A_k$, $i\le k$,

\vspace{2mm}
h8) $|U_{(\Omega_i,\nu_i)}(x,h_k)|<\frac{1}{k^2}$, $\quad x\in X$,$\quad i<k$.

\vspace{2mm}
\noindent
We realize this construction by induction. Applying \lem{L6}, we find an open set $A_1$, with $E\subset A_1\subset E_1$, and a function $h_1(x)$, $x\in X$ such that
\begin{align*}
	&h_1(x)=1,\,x\in A_1,\\
	& 0\le h_1(x)\le 1,\quad x\in X,\\
	&U_n(x,h_1)\to h_1(x) \hbox{ for all } x\in X.
\end{align*}
By \lem{L7} there exists a regular partition $\Omega_1$ of the open set $A_1$ and a function $\nu_1:\Omega_1\to\ZN$ such that
\begin{equation*}
	|U_{(\Omega_1,\nu_1)}(x,h_1)-1|<1,\quad x\in A_1.
\end{equation*}
This gives the base of induction. Now suppose that we have already chosen sets $A_k$ and functions $h_k(x)$, sarisfying h1)-h8) for $k=1,2,\ldots ,p$. Let $c_k(\omega )$, $\omega\in\Omega_k$, be the constants from \e{k5}, corresponding to the regular partition $\Omega_k$ of the set $A_k$, $k=1,2,\ldots ,p$. For any $l=1,2,\ldots,p$ we fix a collection of positive numbers $\{\varepsilon_l(\omega)>0:\,\omega\in \Omega_l\}$, satisfying
\begin{equation}
	\sum_{\omega\in \Omega_l}\varepsilon_l(\omega)<\frac{1}{(p+1)^2}.
\end{equation} 
Then, applying \lem{L11}, for any $l\le p$ we chose an intermediate open set  $C_l$, satisfying
\begin{align}
	&E\subset C_l\subset A_p\cap E_{p+1},\\
	&|C_l\cap \omega |<\frac{\varepsilon_l(\omega)}{c_l(\omega )},\quad \omega \in\Omega_l,\,l=1,2,\ldots ,p.
\end{align}
Denote $C=\cap_{l=1}^pC_l$. Thus, applying \lem{L12}, for any function $f$ satisfying \e{x32},
we get
\begin{align}
	|U_{(\Omega_l,\nu_l)}(x,f)|&\le \sum_{\omega \in\Omega_l}c_l(\omega )|C\cap \omega |\\
	&\le\sum_{\omega \in\Omega_l}c_l(\omega )|C_l\cap \omega |\le \sum_{\omega \in\Omega_l}\varepsilon_l(\omega)< \frac{1}{(p+1)^2},\label{g5}
\end{align}
which holds for every $x\in X$ and $l\le p$. Then using \lem{L6}, we find an open set $A_{p+1}$ and a function $h_{p+1}(x)$, $x\in X$, such that
\begin{align*}
	&E\subset A_{p+1}\subset C\subset A_p\cap E_{p+1},\\
	&\supp h_{p+1}\subset C,\quad h_{p+1}(x)=1,\,x\in A_{p+1},\\
	& 0\le h_{p+1}(x)\le 1,\quad x\in X,\\
	&|U_n(x,h_{p+1})|\le 2^{-p-1} ,\quad x\in (A_p)^c,\\
	&U_n(x,h_{p+1})\to h_{p+1}(x) \hbox{ for all } x\in X.
\end{align*}
Thus we obtain conditions h1)-h5) for $k=p+1$. From \e{g5} we can also get the inequality
\begin{equation*}
	|U_{(\Omega_l,\nu_l)}(x,h_{p+1})|\le \frac{1}{(p+1)^2},\, x\in X, \,l\le p,
\end{equation*}
which implies h8) in the case $k=p+1$.  Then, applying \lem{L7}, we find a locally-finite partition $\Omega_{p+1}$ for $A_{p+1}$ and a mapping $\nu_{p+1}:\Omega_{p+1}\to\ZN$ satisfying h6) and h7) for $k=p+1$. This completes the induction. From h1) we obtain
\begin{equation*}
	E=\cap_{i=1}^\infty A_i.
\end{equation*}
Define
\begin{equation}\label{g6}
	g(x)=\sum_{i=1}^\infty (-1)^{i+1}h_i(x) \text{ if } x\in X\setminus E
\end{equation}
and $g(x)=0$ if $x\in E$. Observe that series \e{g6} converges in $L^1(X)$, since by conditions h1), h2), h3) and \e{x68} it follows that $\|h_k\|_1<\varepsilon 2^{-k+1}$. Thus by the boundedness of the operator $U_n:L^1(X)\to M(X)$ we conclude
\begin{equation}\label{g12}
	U_n(g)=\sum_{i=1}^\infty (-1)^{i+1}U_n(h_i),
\end{equation}
where the series converges uniformly. Clearly, h1) and \e{x68} imply \e{g4}. Fix a point $x\in E^c$. We have
\begin{equation}\label{g8}
	x\in A_{k-1}\setminus A_{k},
\end{equation}
for some $k\ge 1$. Taking into account h2), the latter implies
\begin{align}
	&h_i(x)=0,\quad i>k,\label{g9}\\
	&h_i(x)=1,\quad i<k.\label{g10}
\end{align}
Thus we obtain
\begin{equation}\label{g11}
	g(x)= \sum_{i=1}^k(-1)^{i+1}h_i(x)=\sum_{i=1}^{k-1} (-1)^{i+1}+(-1)^{k+1}h_k(x).
\end{equation}
If $k$ is even, then $g(x)=1-h_k(x)$, for an odd $k$ we have $g(x)=h_k(x)$. From this and property h3) of function $h_k(x)$, we obtain condition \e{g1}. Then condition h4) implies
\begin{equation}\label{x11}
	|U_n(x,h_i)|\le 2^{-i},\quad i>k.
\end{equation}
Thus, using \e{g12}, \e{g11} as well as h5), for $m>k$ we get
\begin{multline*}
	\limsup_{n\to\infty}|U_n(x,g)-g(x)|=\limsup_{n\to\infty} |\sum_{i=m}^\infty (-1)^{i+1}U_n(x,h_i)|\\
	\le \sum_{i=m}^\infty|U_n(x,h_i)|\le\sum_{i=m}^\infty 2^{-i}=2^{-m+1}.
\end{multline*}
Since $m$ can be arbitrarily bigger, this implies condition \e{g2}.
To prove \e{g0}, now suppose $x\in E$. Then we have $x\in A_k$ for all $k=1,2,\ldots $. Using this, for any $k=1,2,\ldots$ we may fix a set $\omega_k\in\Omega_k$ such that $\omega_k\ni x$. By condition h8) we have
\begin{equation}\label{x14}
	|U_{\nu_k(\omega_k)}(x,h_i)|\le \frac{1}{i^2},\quad i>k,
\end{equation}
and form h7) it follows that
\begin{equation}\label{x15}
	|U_{\nu_k(\omega_k)}(x,h_i)-1|\le  \frac{1}{k^2},\quad i\le k.
\end{equation}
Thus, using \e{g12}, we obtain
\begin{multline*}\begin{split}
		|U_{\nu_k(\omega_k)}(x,g)-\sum_{i=1}^k(-1)^{i+1}|\\
		&\le \sum_{i=1}^k|U_{\nu(\omega_k)}(x,h_i)-1|+\sum_{i=k+1}^\infty |U_{\nu(\omega_k)}(x,h_i)|\\
		&\le k\cdot \frac{1}{k^2}+\sum_{i=k+1}^\infty \frac{1}{i^2}<\frac{2}{k}.
	\end{split}
\end{multline*}
Since by h6) we have $\nu_k(\omega_k)\to\infty$ and the sum $\sum_{i=1}^k(-1)^{k+1}$ takes values $0$ or $1$ alternatively, we get \e{g0} for any $x\in E$, finalizing the proof of lemma.
\end{proof}
\section{Proof of theorems}
\begin{proof}[Proof of \trm{T1}]
	Suppose $E$ is a $G_{\delta\sigma}$ null-set and we have $E=\bigcup_{k=1}^\infty E_k$,
	where $E_k$ are $G_\delta$ null-sets and we may additionally suppose that $E_k\subset E_{k+1}$.  Applying \lem{L8}, we attach to each $E_k$ a function $g_k(x)$ such that
	
		\vspace{2mm}
	g1) $\mu(\supp (g_k))<\varepsilon/2^k$,
	
		\vspace{2mm}
	g2) $0\le g_k(x) \le 1$,
	
		\vspace{2mm}
	g3) $U_n(x,g_k)\to g_k(x)$ at every point $x\not\in E_k$,
	
		\vspace{2mm}
	g4) for every $x\in E_k $ we have $\delta(x,g_k)\ge 1/2$, where we use the notation
	\begin{equation}
		\delta(x,g)=\limsup_{n\to\infty}|U_n(x,g)-g(x)|.
	\end{equation}

	\vspace{2mm}
	\noindent
	We claim that
	\begin{equation*}
		f(x)=\sum_{i=1}^\infty\frac{g_i(x)}{(4\varrho+5)^i}
	\end{equation*}
	is our desired function, where $\varrho$ is the constant in \p{C2}. The boundedness of $f$ follows from condition g2), and condition g1) implies $\mu(\supp(f))<\varepsilon$. Chose an arbitrary point $x\in E$. For some  $k$ we have
	\begin{equation*}
		x\in E_k\setminus E_{k-1}.
	\end{equation*}
	This together with g3), g4) implies
	\begin{align}\label{k1}
	&\lim_{n\to\infty}U_n\left(x,g_i\right)=g_i(x)\text{ if }i\le k-1,\\
	&\delta(x,g_k)\ge 1/2,\\
	&\delta(x,g_i)\le \sup_n|U_n(x,g_i)|+|g_i(x)|\le \varrho+1,\quad i=1,2,\ldots.
\end{align}
Thus, using also \p{C2}, we obtain
	\begin{align*}
		\delta(x,f)&\ge \frac{\delta(x,g_k)}{(4\varrho+53)^k}-\sum_{i=k+1}^\infty \frac{\delta(x,g_i)}{(4\varrho+5)^i}\\
		&\ge \frac{1}{2\cdot (4\varrho+5)^k}-\sum_{i=k+1}^\infty \frac{\varrho+1}{(4\varrho+5)^i}\\
		&=\frac{1}{4\cdot (4\varrho+5)^k}>0,
	\end{align*}
	which means that the sequence $U_n(x,f)$ diverges for any $x\in E$. Letting $x\not\in E$, we have $x\not\in E_i$ for any $i\in \ZN$ and so $U_n(x,g_i)\to g_i(x)$ as $n\to\infty$. This implies
	\begin{align*}
		\delta(x,f)&=\delta\left(x,\sum_{i=k+1}^\infty (4\varrho+5)^{-i}g_i\right)\\
		&\le \sup_n\left|U_n\left(x,\sum_{i=k+1}^\infty (4\varrho+5)^{-i}g_i\right)\right|  +\left|\sum_{i=k+1}^\infty (4\varrho+5)^{-i}g_i(x)\right|\\
		&\le (\varrho+1)\sum_{i=k+1}^\infty (4\varrho+5)^{-i}.
	\end{align*}
Since the latter goes to zero, we obtain $\delta(x,f)=0$. Theorem is proved.
\end{proof}
\begin{proof}[Proof of \trm{T2}]
	We will make use the same functions $h_k$ constructed in the proof of \lem{L8}  with an additional bound $\mu(A_k)<\alpha_k$ which can be clearly ensured. Then for small enough numbers $\alpha_k<2^{-k}$ we will have
	\begin{equation}\label{x10}
		f(x)=\sum_{k=1}^\infty h_k(x)\in L_\phi(X),
	\end{equation}
and $\mu(\supp(f))<\varepsilon$. By the construction of the functions $h_k$ the initial set $E$ can be written in the form $E=\cap_kA_k$. Since $\|h_k\|_1<2^{-k+1}$, series \e{x10} converges in $L^1(X)$ and by the boundedness of the operator $U_n:L^1\to M$ we can write
\begin{equation}\label{x12}
	U_n(f)=\sum_{i=1}^\infty U_n(h_i),
\end{equation}
where the series converges uniformly on $X$. Letting $x\in E^c$, we have $x\in A_{k-1}\setminus A_{k}$ for some $k\ge 1$ ($A_0=X$). Then we have \e{g9}, \e{g10} and therefore,
\begin{equation}\label{x13}
	f(x)=\sum_{j=1}^{k}h_j(x)=k-1+h_{k}(x).
\end{equation}
Besides for such a point $x\in A_{k-1}\setminus A_{k}$ we have \e{x11}. Thus using h5) together with \e{x11}, \e{x12} and \e{x13}, for $m>k$ we get
\begin{equation}
	\limsup_{n\to\infty}|U_n(x,f)-f(x)|=\limsup_{n\to\infty} |\sum_{i=m}^\infty U_n(x,h_i)|
	\le \sum_{i=m}^\infty2^{-i}.
\end{equation}
Since $m$ can be arbitrarily bigger, this implies convergence condition b) of the theorem.
To prove condition a), suppose $x\in E$ and so $x\in A_k$ for all $k=1,2,\ldots $. Using this we may fix sets $\omega_k\in \Omega_k$ with $\omega_k\ni x$. Then using  \e{x14} and \e{x15}, we obtain
\begin{align}
		|U_{\nu_k(\omega_k)}(x,f)-k|&\le \sum_{i=1}^k|U_{\nu_k(\omega_k)}(x,h_i)-1|+\sum_{i=k+1}^\infty |U_{\nu_k(\omega_k)}(x,h_i)|\\
		&\le k\cdot \frac{1}{k^2}+\sum_{i=k+1}^\infty \frac{1}{i^2}<\frac{2}{k},
\end{align}
which implies the unboundedly divergence of $U_n(x,f)$ at a point $x\in E$.
\end{proof}
\begin{proof}[Proof of \trm{T3}]
	We construct open sets $G_k$, $k=1,2,\ldots,$ together with a regular partition $\Omega_k$ and a mapping $\nu_k:\Omega_k\to\ZN$, satisfying following conditions.

	\vspace{2mm}
	g1) $E\subset G_k\subset G_{k-1}$, $k\ge 1$ ($G_0=X$),
	
	\vspace{2mm}
	g2) $\min_{\omega\in \Omega_k}\nu_k(\omega)>k$, $k\ge 1$,
	
	\vspace{2mm}
	g3) $|U_{(\Omega_k,\nu_k)}(x,\ZI_{G_i})-1|<\frac{1}{k^2}$, $\quad x\in G_k$, $i\le k$,
	
	\vspace{2mm}
	g4) $|U_{(\Omega_i,\nu_i)}(x,\ZI_{G_k})|<\frac{1}{k^2}$, $\quad x\in X$,$\quad i<k$.
	
	\vspace{2mm}
	\noindent
	We do it by induction. Define the open set $G_1$ arbitrarily. By \lem{L7} there exists a regular partition $\Omega_1$ of $G_1$ and a function $\nu_1:\Omega_1\to\ZN$ such that
	\begin{equation*}
		|U_{(\Omega_1,\nu_1)}(x,\ZI_{G_1})-1|<1,\quad x\in G_1.
	\end{equation*}
	Then suppose that we have already chosen the sets $G_k$ and corresponding partitions $\Omega_k$ for $k=1,2,\ldots,p$.  Applying \lem{L11}, we may choose  open sets  $C_l$, $l=1,2,\ldots,p$ with $E\subset C_l\subset G_p$, satisfying
	\begin{equation*}
		\sum_{\omega\in \Omega_l}c_l(\omega )|C_l\cap \omega |<\frac{1}{(p+1)^2},\quad l=1,2,\ldots ,p,
	\end{equation*}
where $c_l(\omega )$ are the constants from \e{k5}. Then, denoting $G_{p+1}=\cap_{l=1}^p C_l$ and applying \lem{L12}, we obtain the inequality
\begin{equation}\label{x16}
	|U_{(\Omega_l,\nu_l)}(x,\ZI_{G_{p+1}})|\le \sum_{\omega \in\Omega_l}c_l(\omega )|G_{p+1}\cap \omega |\le \frac{1}{(p+1)^2},
\end{equation}
which holds for every $x\in X$ and $l\le p$. This implies g4) for $k=p+1$. Then, applying \lem{L7}, we find a regular partition $\Omega_{p+1}$ for $G_{p+1}$ and a mapping $\nu_{p+1}:\Omega_{p+1}\to\ZN$ satisfying g2) and g3) for $k=p+1$. This completes the induction.  Now we can define our desired set by
\begin{equation}
	G=\bigcup_{i=1}^\infty (G_{2i-1}\setminus G_{2i}).
\end{equation}
One can check,
\begin{equation}
	U_n(x,\ZI_G)=\sum_{k=1}^\infty(-1)^{k+1}U_n(x,\ZI_{G_k}).
\end{equation}
Suppose $x\in E$. Then we have $x\in G_k$, $k=1,2,\ldots $. Using this we may fix $\omega_k\in\Omega_k$ such that $\omega_k\ni x$. By condition g4) we have
\begin{equation}
	|U_{\nu_k(\omega_k)}(x,\ZI_{G_i})|\le \frac{1}{i^2},\quad i>k,
\end{equation}
and form g3) it follows that
\begin{equation}
	|U_{\nu_k(\omega_k)}(x,\ZI_{G_i})-1|\le  \frac{1}{k^2},\quad i\le k.
\end{equation}
Thus we obtain
\begin{align}
		|U_{\nu_k(\omega_k)}(x,f)-\sum_{i=1}^k(-1)^{i+1}|&\\
		&\le \sum_{i=1}^k|U_{\nu(\omega_k)}(x,\ZI_{G_i})-1|+\sum_{i=k+1}^\infty |U_{\nu(\omega_k)}(x,\ZI_{G_i})|\\
		&\le k\cdot \frac{1}{k^2}+\sum_{i=k+1}^\infty \frac{1}{i^2}<\frac{2}{k}
\end{align}
Thus we get divergence of $U_{\nu(\omega_k)}(x,f)$ for any $x\in E$.
\end{proof}
\section{Open problems}
Some open problems listed here one can find also in the papers \cite{Uly1,Uly2,CoLo, Wade}.
	\begin{enumerate}
	\item Find the complete characterization of $D$-sets ($UD$-sets) of ordinary trigonometric series and Fourier series of $L^p$, $1\le p\le \infty$, or $C(\ZT)$ functions. Note that K\"orner \cite{Kor} in 1961 constructed a $G_\delta$-set, which is not a convergence set for any kind of trigonometric series. This example shows that for the partial sums of Fourier or ordinary trigonometric series a pure topological characterization of $C$-$D$ sets may fail (\cite{Uly1,Uly2}). 
	\item The same problems are open also for the Walsh series.
	\item Kahane-Katznelson's \cite{KaKa} theorem analogue for the Walsh system is an open problem(\cite{Wade}). It claims, given null-set $E$, construct a continuous function which Walsh-Fourier series diverges at any $x\in E$ . Concerning to this problem we note that Harris \cite{Har} has proved that for any compact null set $e\subset [0,1]$ there exists a continuous function, whose Walsh-Fourier series diverges at any $x\in e$. As we noted the original proof of \cite{KaKa} essentially uses analytic functions technique, which hardly can be applied in Walsh case. A real functions approach to Kahane-Katznelson's theorem one can find in \cite{Kar4}.
	\item Characterize the sets, which are radial $D$-sets of an univalent function on the unit disc  (\cite{CoLo}).
	\item  It was considered in \cite{Kar4} the exceptional null set problem for the Hilbert transform
	\begin{equation}\label{x66}
		Hf(x)=\lim_{\varepsilon\to 0}\int_{|t-x|>\varepsilon}\frac{f(t)}{x-t}dt.
	\end{equation}
	It is well-known the almost everywhere existence of this limit whenever $f\in L^1(\ZR)$ (see for example \cite {Zyg}, ch. 4.3).  It was proved in \cite{Kar4} that
for any closed null set $e\subset \ZR$ there exists a continuous function $f\in C(\ZR)\cap L^1(\ZR)$ such that the limit in \e{x66} doesn't exists at any point of $e$. We do not know whether the set $e$ in this statement can be an arbitrary null set.
\end{enumerate}


\begin{bibdiv}
	\begin{biblist}
		\bib{Bug}{article}{
		author={Bugadze, V. M.},
		title={On the divergence of Fourier-Haar series of bounded functions on
			sets of measure zero},
		language={Russian, with Russian summary},
		journal={Mat. Zametki},
		volume={51},
		date={1992},
		number={5},
		pages={20--26, 156},
		issn={0025-567X},
		translation={
			journal={Math. Notes},
			volume={51},
			date={1992},
			number={5-6},
			pages={437--441},
			issn={0001-4346},
		},
		review={\MR{1186527}},
		doi={10.1007/BF01262173},
	}
\bib{Buz}{article}{
	author={Buzdalin, V. V.},
	title={Unboundedly diverging trigonometric Fourier series of continuous
		functions},
	language={Russian},
	journal={Mat. Zametki},
	volume={7},
	date={1970},
	pages={7--18},
	issn={0025-567X},
	review={\MR{262757}},
}
\bib{Buz1}{article}{
	author={Buzdalin, V. V.},
	title={Trigonometric Fourier series of continuous functions that diverge
		on a given set},
	language={Russian},
	journal={Mat. Sb. (N.S.)},
	volume={95(137)},
	date={1974},
	pages={84--107, 159},
	issn={0368-8666},
	review={\MR{358197}},
}
\bib{Car}{article}{
	author={Carleson, Lennart},
	title={On convergence and growth of partial sums of Fourier series},
	journal={Acta Math.},
	volume={116},
	date={1966},
	pages={135--157},
	issn={0001-5962},
	review={\MR{199631}},
	doi={10.1007/BF02392815},
}
\bib{Christ}{book}{
	author={Christ, Michael},
	title={Lectures on singular integral operators},
	series={CBMS Regional Conference Series in Mathematics},
	volume={77},
	publisher={Published for the Conference Board of the Mathematical
		Sciences, Washington, DC; by the American Mathematical Society,
		Providence, RI},
	date={1990},
	pages={x+132},
	isbn={0-8218-0728-5},
	review={\MR{1104656}},
}
\bib{CoWe}{book}{
	author={Coifman, Ronald R.},
	author={Weiss, Guido},
	title={Analyse harmonique non-commutative sur certains espaces homog\`enes},
	language={French},
	series={Lecture Notes in Mathematics, Vol. 242},
	note={\'{E}tude de certaines int\'{e}grales singuli\`eres},
	publisher={Springer-Verlag, Berlin-New York},
	date={1971},
	pages={v+160},
	review={\MR{499948}},
}
\bib{CoLo}{book}{
	author={Collingwood, E. F.},
	author={Lohwater, A. J.},
	title={The theory of cluster sets},
	series={Cambridge Tracts in Mathematics and Mathematical Physics, No. 56},
	publisher={Cambridge University Press, Cambridge},
	date={1966},
	pages={xi+211},
	review={\MR{231999}},
}
\bib{Dya}{article}{
	author={D\cprime yachkov, A. M.},
	title={Description of sets of Lebesgue points and summability points of a
		Fourier series},
	language={Russian},
	journal={Mat. Sb.},
	volume={182},
	date={1991},
	number={9},
	pages={1367--1374},
	issn={0368-8666},
	translation={
		journal={Math. USSR-Sb.},
		volume={74},
		date={1993},
		number={1},
		pages={111--118},
		issn={0025-5734},
	},
	review={\MR{1133575}},
}
\bib{Fat}{article}{
	author={Fatou, P.},
	title={S\'{e}ries trigonom\'{e}triques et s\'{e}ries de Taylor},
	journal={Acta Math.},
	volume={30},
	date={1906},
	pages={335--400}
}
\bib{Fow}{article}{
author={Fowler, Thomas},
author={Preiss, David},
title={A simple proof of Zahorski's description of non-differentiability
	sets of Lipschitz functions},
journal={Real Anal. Exchange},
volume={34},
date={2009},
number={1},
pages={127--138},
issn={0147-1937},
review={\MR{2527127}},
}
\bib{Gog}{article}{
	author={Goginava, U.},
	title={On the divergence of Walsh-Fej\'{e}r means of bounded functions on
		sets of measure zero},
	journal={Acta Math. Hungar.},
	volume={121},
	date={2008},
	number={4},
	pages={359--369},
	issn={0236-5294},
	review={\MR{2461440}},
	doi={10.1007/s10474-008-7219-2},
}

\bib{GES}{book}{
	author={Golubov, B.},
	author={Efimov, A.},
	author={Skvortsov, V.},
	title={Walsh series and transforms},
	series={Mathematics and its Applications (Soviet Series)},
	volume={64},
	note={Theory and applications;
		Translated from the 1987 Russian original by W. R. Wade},
	publisher={Kluwer Academic Publishers Group, Dordrecht},
	date={1991},
	pages={xiv+368},
	isbn={0-7923-1100-0},
	review={\MR{1155844}},
	doi={10.1007/978-94-011-3288-6},
}
\bib{Hahn}{article}{
	author={Hahn, H.},
	title={Ueber die Menge der Konvergenzpunkte einer Funktionenfolge},
	journal={Arch. d. Math. u. Phys.},
	volume={28},
	date={1919},
	pages={34--45},
}
\bib{Har}{article}{
	author={Harris, David C.},
	title={Compact sets of divergence for continuous functions on a Vilenkin
		group},
	journal={Proc. Amer. Math. Soc.},
	volume={98},
	date={1986},
	number={3},
	pages={436--440},
	issn={0002-9939},
	review={\MR{857936}},
	doi={10.2307/2046197},
}
\bib{Hau}{book}{
	author={Hausdorff, Felix},
	title={Set theory},
	note={Translated by John R. Aumann, et al},
	publisher={Chelsea Publishing Co., New York},
	date={1957},
	pages={352},
	review={\MR{86020}},
}
\bib{HeWe}{book}{
	author={Hern\'{a}ndez, Eugenio},
	author={Weiss, Guido},
	title={A first course on wavelets},
	series={Studies in Advanced Mathematics},
	note={With a foreword by Yves Meyer},
	publisher={CRC Press, Boca Raton, FL},
	date={1996},
	pages={xx+489},
	isbn={0-8493-8274-2},
	review={\MR{1408902}},
	doi={10.1201/9781420049985},
}
\bib{Hunt}{article}{
	author={Hunt, Richard A.},
	title={On the convergence of Fourier series},
	conference={
		title={Orthogonal Expansions and their Continuous Analogues},
		address={Proc. Conf., Edwardsville, Ill.},
		date={1967},
	},
	book={
		publisher={Southern Illinois Univ. Press, Carbondale, IL},
	},
	date={1968},
	pages={235--255},
	review={\MR{238019}},
}
\bib{KaKa}{article}{
	author={Kahane, Jean-Pierre},
	author={Katznelson, Yitzhak},
	title={Sur les ensembles de divergence des s\'{e}ries trigonom\'{e}triques},
	language={French},
	journal={Studia Math.},
	volume={26},
	date={1966},
	pages={305--306},
	issn={0039-3223},
	review={\MR{199633}},
	doi={10.4064/sm-26-3-307-313},
}
\bib{KaGa}{article}{
	author={Grigoryan, Martin G.},
	author={Kamont, Anna},
	author={Maranjyan, Artavazd A.},
	title={Menshov-type theorem for divergence sets of sequences of localized operators},
	language={Russian, with English and Russian summaries},
	journal={Izv. Nats. Akad. Nauk Armenii Mat.},
	volume={58},
	date={2023},
	number={2},
	pages={46--62},
	issn={0002-3043},
	translation={
		journal={J. Contemp. Math. Anal.},
		volume={48},
		date={2023},
		number={2},
		issn={1068-3623},
	}
}
\bib{Kar1}{article}{
	author={Karagulyan, Grigori A.},
	title={Divergence of general operators on sets of measure zero},
	journal={Colloq. Math.},
	volume={121},
	date={2010},
	number={1},
	pages={113--119},
	issn={0010-1354},
	review={\MR{2725706}},
	doi={10.4064/cm121-1-10},
}
\bib{Kar2}{article}{
	author={Karagulyan, Grigori A.},
	title={On a characterization of the sets of divergence points of
		sequences of operators with the locallyization property},
	language={Russian, with Russian summary},
	journal={Mat. Sb.},
	volume={202},
	date={2011},
	number={1},
	pages={11--36},
	issn={0368-8666},
	translation={
		journal={Sb. Math.},
		volume={202},
		date={2011},
		number={1-2},
		pages={9--33},
		issn={1064-5616},
	},
	review={\MR{2796825}},
	doi={10.1070/SM2011v202n01ABEH004136},
}
\bib{Kar3}{article}{
	author={Karagulyan, Grigori A.},
	title={Complete characterization of sets of divergence points of
		Fourier-Haar series},
	language={Russian, with English and Russian summaries},
	journal={Izv. Nats. Akad. Nauk Armenii Mat.},
	volume={45},
	date={2010},
	number={6},
	pages={33--50},
	issn={0002-3043},
	translation={
		journal={J. Contemp. Math. Anal.},
		volume={45},
		date={2010},
		number={6},
		pages={334--347},
		issn={1068-3623},
	},
	review={\MR{2828809}},
	doi={10.3103/S1068362310060051},
}
\bib{Kar4}{article}{
	author={Karagulyan, Grigori A.},
	title={On exceptional sets of the Hilbert transform},
	journal={Real Anal. Exchange},
	volume={42},
	date={2017},
	number={2},
	pages={311--327},
	issn={0147-1937},
	review={\MR{3721804}},
	doi={10.14321/realanalexch.42.2.0311},
}
\bib{KarKar}{article}{
	author={Karagulyan, G. A.},
	author={Karagulyan, D. A.},
	title={On a characterization of extremal sets of differentiation of
		integrals in $\Bbb R^2$},
	language={Russian, with English and Russian summaries},
	journal={Izv. Nats. Akad. Nauk Armenii Mat.},
	volume={49},
	date={2014},
	number={6},
	pages={83--108},
	issn={0002-3043},
	translation={
		journal={J. Contemp. Math. Anal.},
		volume={49},
		date={2014},
		number={6},
		pages={334--351},
		issn={1068-3623},
	},
	review={\MR{3381401}},
	doi={10.3103/s1068362314060090},
}
\bib{Kat}{book}{
	author={Katznelson, Yitzhak},
	title={An introduction to harmonic analysis},
	series={Cambridge Mathematical Library},
	edition={3},
	publisher={Cambridge University Press, Cambridge},
	date={2004},
	pages={xviii+314},
	isbn={0-521-83829-0},
	isbn={0-521-54359-2},
	review={\MR{2039503}},
	doi={10.1017/CBO9781139165372},
}
\bib{Khe}{article}{
	author={Kheladze, \v{S}. V.},
	title={The divergence everywhere of Fourier-Walsh series},
	language={Russian, with English and Georgian summaries},
	journal={Sakharth. SSR Mecn. Akad. Moambe},
	volume={77},
	date={1975},
	pages={305--307},
	review={\MR{374794}},
}
\bib{Kol}{article}{
	author={Kolesnikov, S. V.},
	title={On the sets of nonexistence of radial limits of bounded analytic
		functions},
	language={Russian, with Russian summary},
	journal={Mat. Sb.},
	volume={185},
	date={1994},
	number={4},
	pages={91--100},
	issn={0368-8666},
	translation={
		journal={Russian Acad. Sci. Sb. Math.},
		volume={81},
		date={1995},
		number={2},
		pages={477--485},
		issn={1064-5616},
	},
	review={\MR{1272188}},
	doi={10.1070/SM1995v081n02ABEH003547},
}
\bib{Kor}{article}{
	author={K\"{o}rner, T. W.},
	title={Sets of divergence for Fourier series},
	journal={Bull. London Math. Soc.},
	volume={3},
	date={1971},
	pages={152--154},
	issn={0024-6093},
	review={\MR{290022}},
	doi={10.1112/blms/3.2.152},
}
\bib{Luk1}{article}{
	author={Luka\v{s}enko, S. Ju.},
	title={The structure of sets of divergence of trigonometric and Walsh
		series},
	language={Russian},
	journal={Dokl. Akad. Nauk SSSR},
	volume={253},
	date={1980},
	number={3},
	pages={528--530},
	issn={0002-3264},
	review={\MR{581404}},
}
\bib{Luk2}{article}{
	author={Lukashenko, S. Yu.},
	title={The structure of divergence sets for Fourier-Walsh series},
	language={Russian, with English summary},
	journal={Anal. Math.},
	volume={10},
	date={1984},
	number={1},
	pages={23--41},
	issn={0133-3852},
	review={\MR{756881}},
	doi={10.1007/BF02115870},
}
\bib{Lun}{article}{
	author={Lunina, M. A.},
	title={The set of points of unbounded divergence of series in the Haar
		system},
	language={Russian, with English summary},
	journal={Vestnik Moskov. Univ. Ser. I Mat. Meh.},
	volume={31},
	date={1976},
	number={4},
	pages={13--20},
	issn={0201-7385},
	review={\MR{415188}},
}
\bib{Lus}{article}{
	author={Luzin, N. N.},
	title={Sur la representation conforme},
	language={Russian},
	journal={Bull. Ivanovo-Vozn. Polytech. Inst.},
	volume={2},
	date={1919},
	pages={77--80},

}
\bib{MaSe}{article}{
	author={Mac\'{\i}as, Roberto A.},
	author={Segovia, Carlos},
	title={Lipschitz functions on spaces of homogeneous type},
	journal={Adv. in Math.},
	volume={33},
	date={1979},
	number={3},
	pages={257--270},
	issn={0001-8708},
	review={\MR{546295}},
	doi={10.1016/0001-8708(79)90012-4},
}
\bib{Mul}{book}{
	author={M\"{u}ller, Paul F. X.},
	title={Isomorphisms between $H^1$ spaces},
	series={Instytut Matematyczny Polskiej Akademii Nauk. Monografie
		Matematyczne (New Series) [Mathematics Institute of the Polish Academy of
		Sciences. Mathematical Monographs (New Series)]},
	volume={66},
	publisher={Birkh\"{a}user Verlag, Basel},
	date={2005},
	pages={xiv+453},
	isbn={978-3-7643-2431-5},
	isbn={3-7643-2431-7},
	review={\MR{2157745}},
}
\bib{PaSh}{article}{
	author={Passenbrunner, M.},
	author={Shadrin, A.},
	title={On almost everywhere convergence of orthogonal spline projections
		with arbitrary knots},
	journal={J. Approx. Theory},
	volume={180},
	date={2014},
	pages={77--89},
	issn={0021-9045},
	review={\MR{3164047}},
	doi={10.1016/j.jat.2013.12.004},
}
\bib{Pir}{article}{
	author={Piranian, George},
	title={The set of nondifferentiability of a continuous function},
	journal={Amer. Math. Monthly},
	volume={73},
	date={1966},
	number={4},
	pages={57--61},
	issn={0002-9890},
	review={\MR{193193}},
	doi={10.2307/2313750},
}
\bib{Pro}{article}{
	author={Prohorenko, V. I.},
	title={Divergent Fourier series with respect to Haar's system},
	language={Russian},
	journal={Izv. Vys\v{s}. U\v{c}ebn. Zaved. Matematika},
	date={1971},
	number={1(104)},
	pages={62--68},
	issn={0021-3446},
	review={\MR{294983}},
}
\bib{Sha}{article}{
	author={Shadrin, A. Yu.},
	title={The $L_\infty$-norm of the $L_2$-spline projector is bounded
		independently of the knot sequence: a proof of de Boor's conjecture},
	journal={Acta Math.},
	volume={187},
	date={2001},
	number={1},
	pages={59--137},
	issn={0001-5962},
	review={\MR{1864631}},
	doi={10.1007/BF02392832},
}
\bib{Sie}{article}{
	author={Sierpinski, W.},
	title={ Sur l'ensemble des points de convergence d'une suite de fonctions continues},
	journal={Fund. Math.},
	volume={2},
	number={1},
	date={1921},
	pages={41--47},
}
\bib{Sim}{book}{
	author={Simon, Leon},
	title={Lectures on geometric measure theory},
	series={Proceedings of the Centre for Mathematical Analysis, Australian
		National University},
	volume={3},
	publisher={Australian National University, Centre for Mathematical
		Analysis, Canberra},
	date={1983},
	pages={vii+272},
	isbn={0-86784-429-9},
	review={\MR{756417}},
}
\bib{Tai}{article}{
	author={Taikov, L. V.},
	title={On the divergence of Fourier series with respect to a re-arranged
		trigonometrical system},
	language={Russian},
	journal={Uspehi Mat. Nauk},
	volume={18},
	date={1963},
	number={5(113)},
	pages={191--198},
	issn={0042-1316},
	review={\MR{156147}},
}
\bib{Uly1}{article}{
	author={Ul\cprime yanov, P. L.},
	title={Divergent Fourier series},
	language={Russian},
	journal={Uspekhi Mat. Nauk},
	volume={12},
	date={1957},
	number={2},
	pages={75--132},
	issn={0042-1316},
}
\bib{Uly2}{article}{
	author={Ul\cprime yanov, P. L.},
	title={A. N. Kolmogorov and divergent Fourier series},
	language={Russian},
	journal={Uspekhi Mat. Nauk},
	volume={38},
	date={1983},
	number={4(232)},
	pages={51--90},
	issn={0042-1316},
	review={\MR{710115}},
}

\bib{Wade}{article}{
	author={Wade, William R.},
	title={Recent developments in the theory of Walsh series},
	journal={Internat. J. Math. Math. Sci.},
	volume={5},
	date={1982},
	number={4},
	pages={625--673},
	issn={0161-1712},
	review={\MR{679409}},
	doi={10.1155/S0161171282000611},
}
\bib{Zah}{article}{
	author={Zahorski, Zygmunt},
	title={Sur l'ensemble des points de non-d\'{e}rivabilit\'{e} d'une fonction
		continue},
	language={French},
	journal={Bull. Soc. Math. France},
	volume={74},
	date={1946},
	pages={147--178},
	issn={0037-9484},
	review={\MR{22592}},
}
\bib{Zah1}{article}{
	author={Zahorski, Zygmunt},
	title={Sur les int\'{e}grales singuli\`eres},
	language={French},
	journal={C. R. Acad. Sci. Paris},
	volume={223},
	date={1946},
	pages={399--401},
	issn={0001-4036},
	review={\MR{18256}},
}
\bib{Zel}{article}{
	author={Zeller, Karl},
	title={\"{U}ber Konvergenzmengen von Fourierreihen},
	language={German},
	journal={Arch. Math.},
	volume={6},
	date={1955},
	pages={335--340},
	issn={0003-889X},
	review={\MR{69302}},
	doi={10.1007/BF01899414},
}
\bib{Zyg}{book}{
	author={Zygmund, A.},
	title={Trigonometric series. 2nd ed. Vols. I, II},
	publisher={Cambridge University Press, New York},
	date={1959},
	pages={Vol. I. xii+383 pp.; Vol. II. vii+354},
	review={\MR{107776}},
}
\end{biblist}
\end{bibdiv}
\end{document}